\documentclass[11pt,leqno,a4paper,oneside]{amsart}

\usepackage{amssymb,amsthm,enumerate}
\usepackage{wasysym}
\usepackage{array,arydshln}
\usepackage{pstricks, pst-plot, pst-node}
\usepackage{hyperref}                         
\usepackage[all]{xy}

\usepackage{tikz}                                 
\usetikzlibrary{angles}
\usetikzlibrary{decorations.markings,calc}
\usetikzlibrary{cd}

\reversemarginpar

\input xy
\xyoption{all}

\entrymodifiers={!!<0pt,0.7ex>+}  

\hypersetup{pdftex,                           
bookmarks=true,
pdffitwindow=true,
colorlinks=true,
citecolor=black,
filecolor=black,
linkcolor=black,
urlcolor=blue,
hypertexnames=true}

\newcommand{\IZ}{{\mathbb{Z}}}

\newcommand{\fp}{{\mathfrak{p}}}     

\newcommand{\cO}{{\mathcal{O}}}

\newcommand{\bB}{{\mathbf{B}}}
\newcommand{\bb}{{\mathbf{b}}}
\newcommand{\bc}{{\mathbf{c}}}

\DeclareMathOperator{\End}{End}               
\DeclareMathOperator{\Aut}{Aut}                  
\DeclareMathOperator{\Res}{Res}               
\DeclareMathOperator{\Ind}{Ind}                  
\DeclareMathOperator{\Inf}{Inf}                    

\DeclareMathOperator{\Irr}{Irr}

\DeclareMathOperator{\soc}{soc}
\DeclareMathOperator{\head}{head}
\DeclareMathOperator{\rad}{rad}
\DeclareMathOperator{\Capp}{Cap}
\DeclareMathOperator{\Exc}{Ex}

\let\lra=\longrightarrow
\let\wt=\widetilde
\let\wh=\widehat

\newtheorem{thm}{Theorem}[section]
\newtheorem{lem}[thm]{Lemma}

\newtheorem{cor}[thm]{Corollary}
\newtheorem{prop}[thm]{Proposition}

\theoremstyle{theorem}
\newtheorem{thma}{Theorem}         

\newtheorem{proba}[thma]{Problem}

\theoremstyle{definition}

\theoremstyle{remark}
\newtheorem{rem}[thm]{Remark}

\newtheorem{nota}[thm]{Notation}

\begin{document}


\title[Trivial source characters in blocks with cyclic defect groups]{Trivial source characters in blocks with cyclic defect groups}

\date{\today}

\author{{Shigeo Koshitani and Caroline Lassueur}}
\address{{\sc Shigeo Koshitani}, Center for Frontier Science, Chiba University, 1-33 Yayoi-cho, Inage-ku, Chiba 263-8522, Japan.}
\email{koshitan@math.s.chiba-u.ac.jp}
\address{{\sc Caroline Lassueur}, FB Mathematik, TU Kaiserslautern, Postfach 3049, 67653 Kaiserslautern, Germany.}
\email{lassueur@mathematik.uni-kl.de}



\begin{abstract}
We describe the ordinary characters of trivial source modules lying in blocks with cyclic defect groups relying on their recent classification in terms of paths on the Brauer tree {\color{black} by G.~Hiss and the second author}. 
In particular, we show how to recover the exceptional constituents of such characters using the source algebra of the block. 
\end{abstract}

\thanks{
The first author was partially supported by 
the Japan Society for Promotion of Science (JSPS), Grant-in-Aid for Scientific Research
(C)19K03416, 2019--2021. The second author gratefully acknowledge financial support by DFG SFB -TRR 195 'Symbolic Tools in
Mathematics and their Application'. The present article is part of Project A18 thereof. The second author also would like to thank the Isaac Newton Institute for Mathematical Sciences, Cambridge, for support and hospitality during the programme \emph{Groups, representations and applications: new perspectives} where work on this paper was undertaken. This work was supported by EPSRC grant no EP/R014604/1.
}

\maketitle


\pagestyle{myheadings}
\markboth{S. Koshitani and C. Lassueur}{Trivial source characters in blocks with cyclic defect groups}

\section{Introduction}
Let $G$ be a finite group, let $p$ be a prime number such that $p\mid |G|$, and let $k$ be an 
algebraically closed field  of characteristic $p\geq 3$. 
Moreover, assume that we are given a $p$-modular system $(K,\cO,k)$ which is large enough for $G$ and all of its subgroups and quotients.  The main aim of this article is to provide a complete description of the ordinary characters of the  trivial source modules lying in a $p$-block $\bB$ with a non-trivial cyclic defect group $D$. 
\par
First of all, it is well-known that trivial source $kG$-modules are liftable to $\cO G$-lattices, and moreover that they lift in a unique way to a trivial source  $\cO G$-lattice. Therefore it is natural to consider the $K$-character afforded by this trivial source lift and consider the following problem. 

\begin{proba}\label{Prob:PA}
Given a $p$-block $B$ of $kG$ with non-trivial cyclic defect groups, describe all  the irreducible constituents of the ordinary characters afforded by the  trivial source lift to $\cO$ of all trivial source $\bB$-modules. 
\end{proba}

Of course a solution to Problem~\ref{Prob:PA} should be given in terms of certain block invariants, which we will determine in due course. 
To begin with, trivial source modules in blocks with cyclic defect groups are classified by \cite{HL19} using the much older classification of the indecomposable modules in such blocks by Janusz \cite{Jan69} through a so-called \emph{path} on the Brauer tree of $\bB$. See~\S\ref{ssec:indecs}. However, as trivial source modules are not invariants of the Morita equivalence class of the block $\bB$, the data of the Brauer tree is not sufficient in general. 
However, they are invariants of the source-algebra-equivalence class of $\bB$. Hence the classification of \cite{HL19} also makes use of further parameters parametrising the source algebra of $\bB$ according to \cite[Theorem 2.7]{Linck96}. Namely a certain endo-permutation $kD$-module, which we will denote $W$, and a sign function that can be determined from the values of the ordinary irreducible characters of $\bB$ at certain elements of $D$.  
\par
Our main result is Theorem~\ref{thm:main}, which  provides us with a solution to Problem~\ref{Prob:PA}, and indeed describes the $K$-character afforded by the trivial source lift of all trivial source $\bB$-modules with arbitrary non-trivial vertices in terms of the above parameters, that  is the Brauer  tree of $\bB$, the $kD$-module $W$ and the sign function. We postpone the precise statement of our main result to Section 7 because it requires to introduce a lot of notation and concepts. However,  more accurately, the irreducible constituents of these characters which are non-exceptional characters of $\bB$ can easily be  determined 
from the aforementioned path associated to the module. On the other hand, the constituents of these characters which are exceptional characters of $\bB$ are much more difficult to describe and all our results in this article focus on this problem. 
Alos notice that trivial source $\bB$-modules with trivial vertices are just the projective indecomposable modules and their characters are well-known. See~\S\ref{ssec:PIMs}.
\par
The paper is organised as follows. In Section~2, we introduce our notation and recall the necessary background results on blocks with cyclic defect groups. In Section~3, we point out some general results about characters of 
trivial source modules in blocks with cyclic defect groups. In Sections 4, 5, and 6, we describe a reduction procedure in three steps bringing us back to computing certain distinguished $K$-characters of the defect group $D$ of the block $\bB$. 
Finally, in Section~7, we recover all characters of all trivial source $\bB$-modules from those of the trivial source $\bb$-modules, where $\bb$ is the Brauer correspondent of $\bB$ in $N_G(D_1)$ and $D_1$ denotes the unique cyclic subgroup of order~$p$ of $D$. This is achieved using a perfect isometry between $\bb$ and $\bB$ induced by a Rickard complex  from Rickard's and Rouquier's work on blocks with cyclic defect groups. (See  \cite[Theorem 11.12.1]{LinckBook}.)
\par
Finally, we note that we leave the case $p=2$ for a further piece of work as it requires further technical computations on characters afforded by endo-permutation lattices with determinant one.


\vspace{6mm}
\section{Notation and quoted results}\label{sec:prelim}

\subsection{General notation}
Throughout, we let $p$ be an odd prime number and   {$G$} a finite group of order divisible by~$p$. 
We let $(K,\cO,k)$ be a $p$-modular system, where  $\cO$ denotes a complete discrete valuation ring of characteristic zero with unique maximal ideal $\frak{p}:=J(\cO)$, algebraically closed residue field $k:=\cO/\fp$ of characteristic $p$, and field of fractions $K=\text{Frac}(\cO)$, which we assume to be large enough for $G$ and its subgroups in the sense that $K$ contains a root of unity {\color{black} of order}  $\exp(G)$, the exponent of~$G$. 
\par
Unless otherwise stated, for $R\in\{\cO,k\}$, $RG$-modules are assumed to be finitely generated left $RG$-lattices, that is free as $R$-modules, and by a block~$\bB$ of~$G$, we mean a block of~$kG$. 
Given a subgroup $H\leq G$, we let $R$  denote the trivial $RG$-lattice, we write $\Res^G_H(M)$ for the restriction of the $RG$-lattice $M$ to $H$, and  $\Ind_H^G(N)$ for the induction of the $RH$-lattice $N$ to $G$. Given a normal subgroup $U$ of $G$, we write $\Inf_{G/U}^{G}(M)$ for the inflation of the $R[G/U]$-module $M$ to $G$. 
If $M$ is a uniserial $kG$-module, then we denote by $\ell(M)$ its  composition length. 
If $P$ is a $p$-group and $Q\leq P$, then  $\Omega_{P/Q}$ denotes  the \emph{relative Heller operator with respect to $Q$}. In other words, if  $M$ is an $RP$-lattice, then $\Omega_{P/Q}(M)$ is the kernel of a $Q$-relative projective cover $P_{P/Q}(M)$ of  $M$. (See  \cite{thevenazRelProj,TheSurvey} for this less standard notion.) In particular $\Omega:=\Omega_{P/\{1\}}$ is the usual Heller operator. 
We denote by $\Irr(G)$ (resp. $\Irr(\bB)$)  the set of irreducible $K$-characters  of $G$ (resp. of the block $\bB$ of $kG$). 
In general, we continue using the notation of \cite{HL19} inasmuch as it was introduced therein and we refer the reader to \cite{LinckBook, ThevenazBook} for further standard notation.\\
\subsection{Trivial source and cotrivial source lattices}\label{ssec:tscots}
An indecomposable $RG$-lattice $M$ with vertex $Q\leq G$ is called a \emph{trivial source $RG$-lattice} if the trivial $RQ$-lattice $R$ is a source of $M$.  
We adopt the convention that  \emph{trivial source $RG$-lattices} are indecomposable by definition.
\par
It is well-known that any trivial source $kG$-module $M$ is liftable to an $\cO G$-lattice. In other words, there exists an $\cO G$-lattice $\widetilde{M}$ such that $M\cong \widetilde{M}/\mathfrak{p}\widetilde{M}$ (see e.g. \cite[Corollary 3.11.4]{BensonBookI}). More accurately, in general, such modules afford several lifts, but, up to isomorphism, there is a unique
one amongst these which is a trivial source $\cO G$-lattice. We denote this trivial source lift by $\wh{M}$ and simply  by $\chi_M$ the $K$-character afforded by $\wh{M}$, that is the character of $K\otimes_{\cO}\wh{M}$.
Character values of trivial source lattices have the following properties.

\begin{lem}[{}{\cite[Lemma II.12.6]{LandrockBook}}]\label{lem:tscharacters}
Let $M$ be a trivial source $kG$-module and let $x$ is a $p$-element of $G$. Then:
\begin{enumerate}
\item[\rm(a)] $\chi_M(x)$ equals the number of indecomposable direct summands of $\Res^{G}_{\langle x \rangle}(M)$ isomorphic to the trivial $k{\langle x \rangle}$-module. In particular,  $\chi_M(x)$ is a non-negative integer.
\item[\rm(b)] $\chi_M(x)\neq 0$ if and only if $x$ belongs to some vertex of $M$.
\end{enumerate}
\end{lem}

\noindent  Following the terminology of \cite[Definition 4.1.10]{HissLux}, an indecomposable $RG$-lattice $M$  with vertex $Q\leq G$ is called a \emph{cotrivial source $RG$-lattice} if the $RQ$-lattice $\Omega(R)$ is a source of $M$. 
It follows that  any cotrivial source $kG$-module $M$ is liftable to an $\cO G$-lattice and affords a unique lift  $\wh{M}$ which is a cotrivial source $\cO G$-lattice. We denote  by  $\chi_{M}$ the character afforded by $K\otimes_{\cO}\wh{M}$.

\vspace{2mm}
\subsection{Blocks with cyclic defect groups}\label{sec:cyclicblocks}
From now on, unless otherwise stated,  we let ${\bf B}$ denote a block of $kG$ with cyclic defect group $D\cong C_{p^n}$ with $n\geq 1$.
For $0\leq i\leq n$, we denote by $D_i$ the unique cyclic subgroup of  order $p^i$ and we set $N_i:=N_G(D_i)$. 
We let $e$ denote the inertial index of $\bB$ and set $m:=\frac{|D|-1}{e}$,  
which we call the \emph{exceptional multiplicity} of $\bB$. 
Then $e\mid p-1$. There are $e$ simple $\bB$-modules $S_1,\ldots,S_e$ and $e+m$ ordinary irreducible characters. We write
$$\Irr(\bB)=\{\chi_1,\ldots,\chi_e\}\sqcup\{\chi_{\lambda} \mid \lambda\in\Lambda\}\,,$$
where $\Lambda$ is an index set with  $|\Lambda|:=m$ (we will give a precise definition of $\Lambda$ in Section~\ref{sec:levelG}). If $m>1$, the characters $\{\chi_{\lambda} \mid \lambda\in\Lambda\}$ denote the exceptional characters of $\bB$, which  all restrict in the same way to the $p$-regular conjugacy classes of~$G$ and the characters  $\chi_1,\ldots,\chi_e$ denote the non-exceptional characters of $\bB$, which are $p$-rational. 
For $\Lambda'\subseteq\Lambda$, we set 
$$\chi^{}_{\Lambda'}:=\sum_{\lambda\in\Lambda'}\chi_{\lambda}\,.$$
We write $\Irr^{\circ}(\bB):=\{\chi_1,\ldots,\chi_e,\chi^{}_{\Lambda}\}$, $\Irr'(\bB):=\{\chi_1,\ldots,\chi_e\}$ and $\Irr_{\Exc}(\bB):=\{\chi_{\lambda} \mid \lambda\in\Lambda\}$. We let $\sigma(\bB)$ denote the Brauer tree  of $\bB$.  
The vertices of $\sigma(\bB)$ are labelled by the ordinary characters in $\Irr^{\circ}(\bB)$ and  the edges of $\sigma(\bB)$ are labelled by the simple $\bB$-modules $S_1,\ldots,S_e$. If $m>1$  the vertex corresponding to $\chi^{}_{\Lambda}$ is called the \emph{exceptional vertex} and is indicated with a black circle in the drawings of $\sigma(\bB)$.  Furthermore, we assume that $\sigma(\bB)$ is given with a planar embedding, determined by specifying, for each vertex of $\sigma(\bB)$, a cyclic ordering of the edges adjacent to this vertex. We use the convention that in a drawing of $\sigma(\bB)$ in the plane, the successor of an edge is the counter-clockwise neighbour of this edge. 
Let now $u$ be a generator of $D_1$. A vertex $\chi\in\Irr^{\circ}(\bB)$ of $\sigma(\bB)$ is said to be positive 
if $\chi(u)>0$ and we write $\chi>0$, whereas it is said to be negative
if $\chi(u)<0$ and in this case we write $\chi<0$. See \cite[\S4.2]{HL19}. 
The character theory of blocks with cyclic defect groups is essentially described by Dade's work \cite{Dad66}. For more detailed information relative to Brauer trees we also refer the reader to \cite[\S 17]{AlperinBook} and \cite[Chapters 1 \& 2]{HissLux}.\\

\vspace{2mm}
\subsection{Indecomposable modules  in  blocks with cyclic defect groups}\label{ssec:indecs}

By results of  Janusz \cite[\S5]{Jan69}, each indecomposable $\bB$-module $X$ which is neither projective nor simple can be encoded using a \emph{path} on $\sigma(\bB)$, which is by definition  a certain connected subgraph of $\sigma(\bB)$. This path may be seen as an ordered sequence $(E_1,\ldots,E_s)$ of edges of  $\sigma(\bB)$, called 
 \emph{top-socle sequence} of $X$, and  where $E_i,E_{i+1}$ have a common vertex for every $1\leq i\leq s-1$, the odd-labelled edges are in the head of $X$ and the even-labelled edge is in the socle of $X$, or conversely, and some edges may be passed twice if necessary.
Moreover, \cite{BleChi} associates to each indecomposable $\bB$-module $X$ two further parameters:  a \emph{direction} $\varepsilon=(\varepsilon_1,\varepsilon_s)$ and a \emph{multiplicity} $\mu$. For  $i\in\{1,s\}$ we set $\varepsilon_i=1$ if $E_i$ is in the head of $X$ and $\varepsilon_i=-1$ if $E_i$ is in the socle of $X$. If $m=1$, then $\mu:=0$. If $m>1$, then $\mu$ corresponds to the number of times that a simple module $E_{j}$ connected to the exceptional vertex  occurs as a composition factor of~$X$ (this is independent of the choice of $E_j$).  The module $X$ is entirely parametrised by its path, direction and multiplicity. We refer to \cite{Jan69,BleChi,HL19} for further details. We will use this classification in order to state our main result in Section~7.

\vspace{2mm}
\subsection{PIMs and hooks in  blocks with cyclic defect groups}\label{ssec:PIMs}
Blocks with cyclic defect groups being Brauer graph algebras (with respect to the Brauer tree), the structure of the PIMs of $\bB$,  can be described as follows (see e.g. \cite[\S 4.18]{BensonBookI}). If $S$ is a simple $\bB$-module, then its projective cover $P_{S}$ is of the form
$$P_{S}=\boxed{\begin{smallmatrix} S\\ Q_a\oplus\, Q_b \\ S\end{smallmatrix}}\,,$$
where $S=\soc(P_{S})=\head(P_{S})$ and the heart of $P_{S}$ is  $\rad(P_{S})/\soc(P_{S})=Q_a\oplus Q_b$ for two uniserial (possibly zero) $\bB$-modules $Q_a$ and $Q_b$. 
Furthermore, if the end vertices of the edge of $\sigma(\bB)$ corresponding to $S$ are labelled by the irreducible characters $\chi_a$ and $\chi_b$, 
then the projective indecomposable character corresponding to $P_{S}$ is $\Phi_{S}=\chi_{a}+\chi_{b}$.
The PIMs of~$\bB$ are precisely the trivial source $\bB$-module with vertex $D_0=\{1\}$.
Furthermore, Green's walk around the Brauer tree \cite{GreenWalk} provides us with a description of certain distinguished  indecomposable $\bB$-modules, called \emph{hooks} in \cite{BleChi}, \cite{HN12} and \cite{HL19}.   More precisely, following  \cite[\S 2.3]{BleChi}, the uniserial modules of the form 
$$H_a:=\boxed{\begin{smallmatrix}S\\Q_a\end{smallmatrix}}\qquad\text{ and }\qquad H_b:=\boxed{\begin{smallmatrix}S\\Q_b\end{smallmatrix}}$$
for a simple $\bB$-module $S$ are called the \textit{hooks} of $\bB$.  
The vertices of such modules are the defect groups of $\bB$, and any lift of $H_a$ 
affords the character $\chi_a$ and any lift of $H_b$ 
affords the character~$\chi_b$.

\vspace{2mm}
\subsection{Trivial source modules in blocks with cyclic defect groups, reduction to $kD$}\label{ssec:tscyclicblocks}
We quickly recall the principal steps in the \cite{HL19} classification of trivial source $\bB$-modules. \\

First of all,  up to isomorphism,  the set of trivial source $\bB$-modules with a given vertex $D_i\leq D$ ($1\leq i\leq n$) form exactly one $\Omega^2$-orbit  $\{\Omega^{2a}(M)\mid 0\leq a\leq e-1 \}$ of $\bB$-modules, 
where $M$ is a given trivial source $\bB$-module with vertex $D_i$, and the set of cotrivial source modules with vertex $D_i$ forms the $\Omega^2$-orbit $\{\Omega^{2a+1}(M)\mid 0\leq a\leq e-1 \}$. This follows from the fact that the trivial $kD_i$-module is periodic of period~$2$. Now, the trivial source $\bB$-modules are classified by \cite[Theorem~5.4]{HL19} in terms of their \emph{path} on the Brauer tree $\sigma(\bB)$.  Our aim is to use this classification in order to determine the $K$-characters of their trivial source lift to $\cO$.
 More precisely, we are going to go through the reduction to $kD$ used in \cite{HL19} to recover the trivial source $\bB$-modules in order to compute their ordinary characters, as well.\\

 Thus, throughout we let  $\bb$ denote the Brauer correspondent of $\bB$ in~$N_1$, $\bc$ be a block of $C_G(D_1)$ covered by $\bb$ and $T(\bc)$ be the inertia group of $\bc$ in $N_1$, and $A$ denote a source algebra of~$\bc$. Then $D$ is a defect group of the blocks $\bb$, $\bb'$ and $\bc$. The block $\bc$ is nilpotent, whereas the blocks $\bb$ and $\bb'$ have inertial index $e$ and exceptional multiplicity $m$. 
 Furthermore, we let $W$ denote the indecomposable capped endo-permutation $kD$-module parametrising  the block $\bB$ up to source-algebra equivalence. (See \cite[Theorem~2.7]{Linck96}.) Concretely, $W$ may be thought of either as a source of the simple $\bb$-modules, or as a source of the unique simple $\bc$-module. Hence $D_1$ acts trivially on $W$.\\

First, we recall that if $P$ is a finite $p$-group, then a $kP$-module $M$ is called \emph{endo-permutation} if its  $k$-endomorphism algebra $\End_k(M)$ is a permutation $kP$-module. Moreover, an endo-permutation $kP$-module $M$ is said to be \emph{capped} if it has an indecomposable direct summand with vertex $P$, which we is usually denoted by $\Capp(M)$. For further details, we refer the reader to the survey \cite{TheSurvey}. Endo-permutation modules over abelian $p$-groups were classified 
by Dade {\color{black}\cite{DADE78a,DADE78b}}. This classification -- see  \cite{TheSurvey} and \cite[\S 4.5]{HL19}  --  allows us to write the module $W$ parametrising the source-algebra of $\bB$ as follows:

\begin{nota}\label{nota:W}
The $kD$-module $W$ has the form  
$$W=\Omega_{D/D_0}^{a_0}\circ\Omega_{D/D_1}^{a_1}\circ\cdots\circ\Omega_{D/D_{n-1}}^{a_{n-1}}(k)$$
with  $a_i\in\{0,1\}$ for each $0\leq i\leq n-1$. Moreover, 
we assume that $i_0<i_1<\ldots< i_s$ are the indices such that $a_{i_0}=\ldots=a_{i_s}=1$ and $a_i=0$ if $i\in\{0,\ldots,n-1\}\setminus\{i_0,\ldots, i_s\}$, and we set $s:=-1$ if $W=k$. We may in fact also assume that  $a_0=0$, since $D_1$ acts trivially on~$W$. Hence, in the sequel,  we will write
$$W=W(0<i_0<i_1<\ldots< i_s<n)\,.$$
Furthermore, for each $1\leq i\leq n$ we set $\ell_i:=\dim_k\left( \Capp(\Res^D_{D_i}(W)) \right)$, which can be explicitly computed as $\ell_i=\sum_{0\leq i_j<i}(-1)^j p^{i-i_j}+(-1)^{|\{j\mid 0\leq i_j<i\}|}$ (see \cite[Theorem~5.1]{HL19}).
\end{nota}

The reduction to $kD$ works as follows.  
 Firstly, the Green correspondence with respect to $(G,N_1;D)$, which we denote by $f^{-1}$ (upwards) and $f$ (downwards), commutes with the Brauer correspondence and preserves vertices and sources, hence trivial source modules. 
Secondly, the theorem of Fong-Reynolds provides us with a source-algebra equivalence between $\bb$ and $\bb'$, which obviously preserve trivial source modules. Thirdly, we can then reduce to $\bc$, which is a nilpotent block, via induction/restriction using Clifford's theory, which also preserve vertices and sources.  
We can then  further reduce to $kD$ via two  Morita equivalences (see \S\ref{ssec:PuigForBlockc}  for details):
 $$
\begin{tikzcd}[column sep=large]
 kD \arrow[r, leftrightarrow,  "\sim_M", "W\otimes_k-"']   &  A    \arrow[r, leftrightarrow, "\sim_M"]   &     \bc
\end{tikzcd}
$$
If $M$ is an indecomposable $\bc$-module, then we simply call \emph{Morita correspondent of $M$}, the Morita correspondent of $M$ in $kD$ under the composition of these two Morita equivalences. 

\begin{lem}[{}{\cite[Lemma 4.6]{HL19}}]\label{lem:tsc}
Let $M$ be the unique trivial source $\bc$-module with vertex $1<D_i\leq D$. Then the Morita correspondent of $M$ is the $kD$-module
$$U_{D_i}(W):=\left(\Ind_{D_i}^D\circ\Capp\circ\Res^D_{D_i}\right)\!(W)$$
and satisfies $\dim_k(U_{D_i}(W))=\ell_i\cdot p^{n-i}$.
\end{lem} 
 
\noindent  We emphasise that $U_{D_i}(W)$ is not a trivial source module any more in general.  We refer the reader to \cite{HL19} for proofs and further details on this subsection.

 \vspace{2mm}
\subsection{On Puig's characterisation of source algebras of nilpotent blocks}\label{ssec:PuigForBlockc} The block $\bc$ lifts uniquely to a block of  $\cO C_G(D_1)$, say $\tilde{\bc}$.
Then Puig's theorem on nilpotent blocks  (see \cite[Theorem 8.11.5, Corollary 8.11.11]{LinckBook}) states that  any source algebra $\tilde{A}$ of the block $\tilde{\bc}$ is isomorphic to 
$$\tilde{S}\otimes_{\cO} \cO D$$
 as interior $D$-algebra,
where $\tilde{S}:=\End_{\cO}(\tilde{W})$ for an indecomposable endo-permutation $\cO D$-module $\tilde{W}$ with vertex $D$  and of determinant $1$.  Moreover,  if $W:=k\otimes_{\cO}\tilde{W}$, then any source algebra $A$ of the block~$\bc$ is  isomorphic to  
 $$S\otimes_k kD$$
 as interior $D$-algebra, 
where $S:=\End_k(W)$ and $W$ is also an indecomposable endo-permutation $kD$-module with vertex $D$. 
Moreover, the module $W$ can be explicitly realised as a source of the unique simple $\bc$-module $V$, and 
hence also as a source of the simple $\bb$-modules. As  $D_1\trianglelefteq C_G(D_1)$ it follows 
from  Clifford's theory  that $D_1$ acts trivially on $V$, hence also on $W$.

 \par
 More precisely, we have Morita equivalences:
 $$
\begin{tikzcd}[column sep=large]
\Phi_k:\,\,\, kD \arrow[r, leftrightarrow,  "\sim_M", "W\otimes_k-"']   &  A    \arrow[r, leftrightarrow, "\sim_M"]   &     \bc
\end{tikzcd}
$$
The first one is obtained by tensoring over $k$ with $W$ viewed as an $S$-module. In other words, an arbitrary indecomposable $A$-module is of the form $W\otimes_k U$, where $U$ is an indecomposable $kD$-module. For the second one let $i\in \bc^D$ be a source idempotent of $\bc$ such that $A=ikGi$. Then the $(\bc,A)$-bimodule $\bc i$ and the $(A,\bc)$-bimodule $i\bc$ realise a Morita equivalence between $A$ and $\bc$, where an indecomposable $\bc$-module $M$ corresponds to the $A$-module $iM$.  See \cite[(38.2)]{ThevenazBook}.  There are also two Morita equivalences analogously defined over $\cO$:
 $$
\begin{tikzcd}[column sep=large]
\Phi_{\cO}:\,\,\,  \cO D  \arrow[r, leftrightarrow,  "\sim_M", "\widetilde{W}\otimes_{\cO}-"']   &  \widetilde{A}    \arrow[r, leftrightarrow, "\sim_M"]   &     {\widetilde\bc}
\end{tikzcd}
$$
Tensoring everything with $K$ we write $W_K:=K\otimes_\cO \widetilde{W}$,  $S_K:=K\otimes_\cO S=\End_K(W_K)$, so that there are Morita equivalences
 $$
\begin{tikzcd}[column sep=large]
\Phi_{K}:\,\,\,   KD  \arrow[r, leftrightarrow,  "\sim_M", "W_K\otimes_{K}-"']   &  K\otimes_\cO\widetilde{A}\cong S_K\otimes_KKD  \arrow[r, leftrightarrow, "\sim_M"]   &    K\otimes_\cO{\widetilde\bc}\,.
\end{tikzcd}
$$
These in turn induce 
bijections \enlargethispage{5mm}
$$
\begin{tikzcd}[column sep=large]
\Gamma_{K}:\,\,\,  \Irr_K(D)  \arrow[r, rightarrow,  "\sim"]   &  \Irr_K(K\otimes_{\cO}\wt{A})  \arrow[r, rightarrow, "\sim"]   &  \Irr_K(\bc)
\end{tikzcd}
$$
between the sets of $K$-characters of $D$ and $\bc$, where $\Irr_K(K\otimes_{\cO}\wt{A})=\{\rho_{ \widetilde{W} }\cdot \lambda\mid \lambda\in \Irr_K(D) \}$ (see \cite[(52.6)]{ThevenazBook}) and $\rho_{\widetilde{W}}$ is the $K$-character afforded by $\widetilde{W}$.
By abuse of notation, we also denote by $\Gamma_{K}$ its $\IZ$-linear extension to $\IZ\Irr_K(D)$. 
Finally, we may use these bijections to label the $K$-characters of $\bc$. In other words, we may write
$$\Irr_K(\bc)=\{\psi_{\lambda}\mid \lambda\in \Irr_K(D)\}\quad \text{where }\psi_{\lambda}:=\Gamma_K(\lambda)\,.$$
See \cite[(52.8)(a) and its proof]{ThevenazBook}.


\vspace{8mm}
\section{Ordinary characters of trivial source modules: general results}


 \subsection{PIMs and hooks}
 
 We start with two elementary cases, which already let us rule out the case in which the exceptional multiplicity is one.

\begin{lem}\label{lem:PIMhooksSimples}{\ }
\begin{enumerate} 
\item[\rm(a)] If $M$ is a trivial source $\bB$-module with vertex $D_0=\{1\}$, then $M$ is a PIM. In other words, there exists a simple $\bB$-module $S$ such that $M=P_S$ and 
$$\chi_M=\chi_a+\chi_b\,,$$ where $\chi_a$ and $\chi_b$ label the vertices of $\sigma(\bB)$ adjacent to the edge labelled by $S$.
\item[\rm(b)] If a hook $M$ of $\bB$ is a trivial source module, then $\chi_{M}\in\Irr^{\circ}(\bB)$ and  $\chi_M(x)> 0$ for each $x\in D$. 
\end{enumerate}
\end{lem}

\begin{proof}
\begin{enumerate}
\item[\rm(a)] Well-known. See \S\ref{ssec:PIMs}. 
\item[\rm(b)] See \S\ref{ssec:PIMs} and Lemma~\ref{lem:tscharacters}. 
\end{enumerate}
\end{proof}

\begin{cor}[The case $m=1$]
If $m=1$, then the 
trivial source 
 $\bB$-modules are precisely the PIMs and the hooks of $\bB$ whose Green correspondents in $\bb$ are simple.  Their $K$-characters are described by Lemma~\ref{lem:PIMhooksSimples}(a) and ~(b). 
\end{cor}

\begin{proof}
If $m=1$, then $e=|D|-1$, hence $D=D_1$ is cyclic of order $p$.  The 
 trivial source
 $\bB$-modules with vertex $D_0$ are the PIMs of $\bB$.  Now as $D=D_1$, and hence $N_G(D)=N_G(D_1)$, the simple  $\bb$-modules, which all have vertex $D$, are trivial source modules by Clifford theory. These are then all the trivial source $\bb$-modules with vertex $D$ and  their Green correspondents in $\bB$ must be exactly the trivial source $\bB$-modules with vertex $D$. The claim follows.
\end{proof}

$$\boxed{\text{Thus, henceforth, we may  assume that  $m>1$.}}$$
\bigskip 

\subsection{Arbitrary vertices}

Next we state some general facts about characters of trivial source modules with arbitrary vertices. 

\begin{nota} \label{nota:chiM}
If $M$ is a trivial source $\bB$-module (with an arbitrary vertex), then the $K$-character $\chi_M$ afforded by $\wh{M}$, the trivial source lift of $M$, satisfies
$$\langle\chi_M,\chi\rangle_G\in\{0,1\} \quad \text{for all } \chi\in\Irr(\bB)$$
e.g. by  \cite[Theorem~A.1(d)]{HL19} if $e>1$, whereas it is obvious if $e=1$. Therefore, throughout we shall write
$$\chi_M=\Psi_M+\Xi_M$$
where $\Psi_M$ is a sum (possibly empty) of pairwise distinct non-exceptional irreducible characters in $\Irr'(\bB)$ and $\Xi_M$ is a sum  (possibly empty) of pairwise distinct non-exceptional irreducible characters in $\Irr_{\Exc}(\bB)$.  We call $\Psi_M$ the \emph{non-exceptional part} of $\chi_M$ and $\Xi_M$ the \emph{exceptional part} of $\chi_M$. By the above $\Xi_M$ is of the form $\Xi_M=\chi^{}_{\Lambda'}$ for sum $\Lambda'\subseteq\Lambda$ and $|\Lambda'|=\langle\Xi_M,\Xi_M\rangle_G$. 
\end{nota} 

The irreducible constituents of the character $\Psi_M$ are entirely determined by \cite[Theorem~5.4]{HL19} together with \cite[Theorem~A.1]{HL19}. Hence our main aim task is to determine the constituents of $\Xi_M$ in the general case. \\

We start by proving that for a non-projective trivial source module $M$ which is not a hook, $\Xi_M$ is invariant under $\Omega^2$, hence depend only on the order  of the vertices. 

\begin{lem}\label{lem:=exc}
Assume $e>1$ and $m>1$. 
Let $M$ be  a non-projective trivial source $\bB$-module which is not a hook. Then
$$\Xi_{\Omega^{2a+1}(M)}=\chi^{}_{\Lambda}-\Xi_{M}\quad\text{ and }\quad\Xi_{\Omega^{2a}(M)}=\Xi_{M}$$
for each $0\leq a <e$. In particular, if $M$ and $N$ are two non-isomorphic trivial source $\bB$-modules with 
{\color{black}a common vertex $D_i$ (where $1\leq i\leq n$) and  which are not hooks,  then $\Xi_M=\Xi_N$. 
}
\end{lem}

\begin{proof}
As recalled at the beginning of the section $\Omega^{2a+1}(M)$ is a cotrivial source $\bB$-module for each $0\leq a <e$ and $\Omega^{2a}(M)$ is a trivial source $\bB$-module for each $0\leq a <e$. 
Since $M$ is not a hook, by \cite[Theorem A.1]{HL19}, the head of $M$ has exactly one constituent corresponding to a simple  $\bB$-module $E$ labelling an edge  of $\sigma(\bB)$ adjacent to the exceptional vertex. Therefore the multiplicity of $P(E)$ as a direct summand of $P(M)$ is one and 
$$\chi_{P(M)}=\Theta + \chi^{}_{\Lambda}\,,$$
where by Lemma~\ref{lem:PIMhooksSimples} all the irreducible constituents of $\Theta$ are in $\Irr'(\bB)$.\par Now if $\begin{tikzcd}[cramped, sep=small] \Omega(M) \arrow[r, hook] & P(M) \arrow[r, two heads] & M  \end{tikzcd}$ is a projective cover of $M$, then $\begin{tikzcd}[cramped, sep=small] \wh{\Omega(M)} \arrow[r, hook] & P(\wh{M}) \arrow[r, two heads] & \wh{M}  \end{tikzcd}$ is a projective cover of $\wh{M}$. It follows that in the Grothendieck group of $KG$ we have 
$$\chi_{\Omega(M)}=\chi_{P(M)}-\chi_M= \Theta + \chi^{}_{\Lambda}- \Psi_M-\Xi_M\,,$$
hence $\Xi_{\Omega(M)}=\chi^{}_{\Lambda}-\Xi_{M}$. The same argument applied to $\Omega(M)$ yields  $\Xi_{\Omega^2(M)}=\chi^{}_{\Lambda}+\Xi_{M}$ and the first claim follows by iteration of this argument. \par
The second claim is then straightforward{\color{black},} 
because $D_i$ is a common vertex of $M$ and $N$, there exists an integer  $1\leq a < e$ such that $N\cong \Omega^{2a}(M)$. 
\end{proof}

With these general results, we can proceed in the next four sections in  four successive steps to recover the characters of the trivial source $\bB$-modules in the general case.



\vspace{8mm}
\section{Step 1: Characters of the Morita correspondents in $kD$}\label{sec:kD}

In this section, we compute the $K$-characters of the Morita correspondents  in $kD$ of the trivial source $\bc$-modules, that is of the modules $U_Q(W)$ ($1< Q\leq D$). To achieve this aim, in an intermediary step, we describe the character of the capped endo-permutation $kD$-module $W$. 
\bigskip

\subsection{Representation theory of $D$}

The representation theory of $kD$ is well-known. In particular, $kD$ has finite representation type. Letting $u$ denote a generator of $D$, there is a $k$-algebra isomorphism $kD\cong k[X]/(X-1)^{p^n}$ mapping $u \mapsto \overline{X}:=X+(X-1)^{p^n}$, and for $1\leq r\leq p^n$ the module $M_r:=k[X]/(X-1)^r$ is the unique indecomposable $kD$-module of $k$-dimension $r$. In fact, these form a complete set of representatives of the isomorphism classes of  indecomposable  $kD$-modules, and are all uniserial. We refer the reader to \cite[Exercises 17.2 and 28.3]{ThevenazBook} for further details.
\par Similarly all indecomposable modules over all subgroups and quotients of $D$ are parametrised by their $k$-dimension. Thus,  when the module structure is clear from the context, we use the same notational conventions for quotients and subgroups of $D$ as for $D$ itself. \cite[Exercises 17.2 and 28.3]{ThevenazBook} in particular tell us that (endo-permutation) $kD$-modules can be understood inductively from proper subgroups making repetitive use of the Heller operator and inflation. Now, clearly, if $0\leq i\leq n-1$ and $1\leq r\leq p^{n-i}-1$, then $D_i=\langle u^{p^{n-i}}\rangle$ acts trivially on $M_r$, so that $M_r$ may be considered as $k[D/D_i]$-module, namely by abuse of notation we may write $M_r=\Inf_{D/D_i}^D(M_r)$.

\begin{nota}\label{character-N}
We let $\zeta\in K^\times$ denote a primitive $p^n$-th root of unity in $K$.  Then
$$\Irr_K(D)=\{ \lambda^D_\kappa:D=\langle u\rangle \lra K^{\times}, u\mapsto \zeta^\kappa\mid \kappa\in\IZ\text{ and } 0\leq \kappa\leq p^{n}-1\}\,.$$
Then $\lambda^D_0=1_D$ is the unique {\it non-exceptional} $K$-character of~$D$ and 
$$\{\lambda^D_{\kappa}\mid 1\leq \kappa\leq p^{n}-1\}=\Irr_{\Exc}(kD)$$
(see e.g. \cite{Dad66}).  Clearly (see  e.g. Lemma~\ref{lem:PIMhooksSimples}(a)) the projective indecomposable module $kD$ affords the $K$-character 
$$\chi_{kD}=\sum_{\kappa=0}^{p^n-1}\lambda_{\kappa}^{D}\,.$$
\end{nota}

\begin{rem}\label{rem:inf}
Given $0\leq i\leq n-1$, the character $\lambda^D_\kappa$ may be seen as inflated from a character of $D/D_i$ if and only if $D_i\leq\ker\lambda_\kappa^{D}$.   
Thus,
\begin{equation*}
\begin{split}
   \Inf_{D/D_i}^D \left(\Irr_K(D/D_i)\right)  & =\{ \Inf_{D/D_i}^D (\lambda_{\nu}^{D/D_i})\mid 0\leq \nu\leq p^{n-i}-1 \}  \\
                                                               & = \{ \lambda^D_\kappa\mid 0\leq \kappa\leq p^n-1 \text{ and }p^i|\kappa \}\,.                   
\end{split}
\end{equation*} 
\end{rem}

\vspace{2mm}
\subsection{Character of the endo-permutation $kD$-module $W$}\label{ssec:cappedEP}

In view of \S\ref{ssec:tscyclicblocks} and \S\ref{ssec:PuigForBlockc}, we first need to describe the $K$-characters of the capped endo-permutation $\cO D$-lattices of determinant 1 lifting a module of the form $W_D(a_0,\ldots,a_{n-1})$ with $a_0=0$. We recall that given an $\cO D$-lattice~$L$, we may consider the composition of the underlying representation of $D$ with the determinant homomorphism  $\det:\text{GL}(L)\lra \cO^{\times}$. This is a linear character of~$D$, called the \emph{determinant} of $L$.  If 
this character is the trivial character, then it is  said that~$L$ is an $\cO D$-lattice \emph{of determinant}~{\color{black}1}.\par

\begin{lem}\label{lem:det1}{\ }
\begin{enumerate}
\item[\rm(a)] Any permutation $\cO D$-lattice  has determinant $1$.
\item[\rm(b)] If $N$ is an indecomposable capped endo-permutation $kD$-module, then $N$ is liftable to an $\cO D$-lattice, and amongst all possible lifts of $N$ there is  a unique lift $\wt{N}$ with determinant~1. 
\end{enumerate}
\end{lem}

\begin{proof}
\begin{enumerate}
\item[\rm(a)] This holds because $p$ is odd. See \cite[Lemma 3.3(a)]{LT19}.
\item[\rm(b)] It is well-known that all modules belonging to a cyclic block with inertial index 1 are liftable. The claim about the determinant holds by \cite[(28.1)]{ThevenazBook}. 
\end{enumerate}
\end{proof}

\begin{nota}
If $N$ is an indecomposable capped endo-permutation $kD$-module, then we denote by $\chi_N$ the $K$-character of its  unique lift of determinant $1$. 
Notice that the unique indecomposable capped endo-permutation $kD$-module which is also a trivial source module is the trivial module~$k$. Its trivial source lift is the trivial $\cO D$-lattice $\cO$, which obviously has determinant 1, hence the above notation agrees with the notation chosen for the character of the trivial source lift.
\end{nota}

\begin{lem}\label{lem:MD/Di}
If $\bB=kD$, then there is a unique trivial source module with vertex $D_i$ for each $1\leq i\leq n$, namely $\Ind_{D_i}^{D}(k)=\Inf_{D/D_i}^{D}(k[D/D_i])=M_{|D/D_i|}$, which we may also see as the permutation $kD$-module $k[D/D_i]$ with stabiliser $D_i$.
\begin{enumerate}
\item[\rm(a)] The trivial source lift of $k[D/D_i]$ is $\cO[D/D_i]$ and has determinant $1$.
\item[\rm(b)] We have
$$\chi_{M_{|D/D_i|}}=\sum_{\substack{0\leq \kappa \leq  p^n -1\\ p^i | \kappa}}  \lambda^D_\kappa\,.$$
\end{enumerate}
\end{lem}

\begin{proof}
The module  $\Ind_{D_i}^{D}(k)$ is indecomposable, hence is the unique trivial source $kD$-module with vertex $D_i$ and has diemnsion $|D/D_i|$. It is also clear that $\Ind_{D_i}^{D}(k)$ is the inflation from $D/D_i$ to $D$ of the projective indecomposable $k[D/D_i]$-module $k[D/D_i]$.
Now the trivial source lift of $k[D/D_i]$ is $\Ind_{D_i}^{D}(\cO)=\cO[D/D_i]$, which has determinant 1 by Lemma~\ref{lem:det1}(a). 
Hence, it follows from the above and Remark~\ref{rem:inf} that
$$\chi_{M_{|D/D_i|}}=\chi_{\Inf_{D/D_i}^{D}(k[D/D_i])}=\Inf_{D/D_i}^{D}
\left(     \sum_{\kappa=0}^{|D/D_i|-1}\lambda_{\kappa}^{D/D_i}
\right)  =\sum_{\substack{0\leq \kappa \leq  p^n -1\\ p^i | \kappa}}  \lambda^D_\kappa\,.$$
\end{proof}

Now we recall that $M$ is an endo-permutation $kP$-module if and only if $\Omega_{P/Q}(M)$ is an endo-permutation $kP$-module. (See \cite{TheSurvey}.)\\

\begin{lem}\label{lem:OmegaD/Di}
Let $0\leq i\leq n-1$ and $1\leq r\leq p^{n-i}-1$. Then the following holds:
\begin{enumerate}
\item[\rm(a)] $\Omega_{D/D_i}(M_r)= \Inf_{D/Di}^D\left(\Omega(M_r)\right)=M_{|D/D_i|-r}$;
\item[\rm(b)] Let $N$ be an indecomposable capped endo-permutation $kD$-module and let $\wt{N}$ denote its unique lift  with determinant $1$. If $1\leq i\leq n$ and $\dim_k(N)\leq p^{n-i}-1$, then  $\Omega_{D/D_i}(\wt{N})$ is the unique lift of determinant $1$ of $\Omega_{D/D_i}(N)$. 
\item[\rm(c)] $\Omega_{D/D_i}(k) = M_{|D/D_i|-1}$, its lift of determinant 1 is $\Omega_{D/D_i}(\cO)$ and affords the $K$-character  
$$\chi_{\Omega_{D/D_i}(k)}
 =\Big( \sum_{\substack{0\leq \kappa \leq  p^n -1\\ p^i | \kappa}}\lambda^D_\kappa\; 
    \Big) \ - \ \lambda^D_0\ = \sum_{\substack{1\leq \kappa \leq  p^n -1\\ p^i | \kappa}}\lambda^D_\kappa\,.$$
\end{enumerate}
\end{lem}

\begin{proof}
\begin{enumerate}
\item[\rm(a)] 
Let $M_r$ be the unique $k[D/D_i]$-module of dimension $r$ and let 
$$\begin{tikzcd}[cramped, sep=small]
 0\arrow[r] & \Omega(M_r) \ar[r] & P(M_r) \ar[r] & M_r\ar[r] & 0
 \end{tikzcd}$$
 be a projective cover of $M_r$. Because $D/D_i$ is a $p$-group and $M_r$ is uniserial, the head of $M_r$ is the trivial 
$k[D/D_i]$-module and it follows that  $P(M_r)=k[D/D_i]$, i.e. the unique projective indecomposable $k[D/D_i]$-module. Moreover,  $\Omega(M_r)$ is indecomposable because $M_r$ is indecomposable. Therefore, taking inflation to $D$ yields a $D_i$-relative projective cover of $M_r$ seen as a $kD$-module
 $$\begin{tikzcd}[cramped, sep=small]
 0\arrow[r] & \Inf_{D/D_i}^D\left(\Omega(M_r)\right) \ar[r] & \Inf_{D/D_i}^D\left(P(M_r)\right) \ar[r] & \Inf_{D/D_i}^D(M_r)\ar[r] & 0
 \end{tikzcd}$$
 since $\Inf_{D/D_i}^D(M_r)=M_r$. Thus, $P_{D/D_i}(M_r)=\Inf_{D/D_i}^D\left(P(M_r)\right)=k[D/D_i]$, i.e. the indecomposable permutation $kD$-module with stabiliser $D_i$, and 
 $$\Omega_{D/D_i}(M_r)= \Inf_{D/D_i}^D\left(\Omega(M_r)\right)\,.$$ 
 Moreover, 
 $$\dim_k(\Omega_{D/D_i}(M_r))= \dim_k k[D/D_i]-\dim_k M_r = |D/D_i|-r\,.$$
\item[\rm(b)] For the second claim, let 
$$\begin{tikzcd}[cramped, sep=small] 0\arrow[r] & \Omega_{D/D_i}(N) \ar[r] &P_{D/D_i}(N)  \ar[r] & N\ar[r] & 0\end{tikzcd}$$
be a $D_i$-relative projective cover of $N$. Then, by the arguments of the proof of (a),  $P_{D/D_i}(N)=k[D/D_i]$ is a permutation $kD$-module and this short exact sequence lifts to a $D_i$-relative projective cover of $\wt{N}$:
$$\begin{tikzcd}[cramped, sep=small] 0\arrow[r] & \Omega_{D/D_i}(\wt{N}) \ar[r] &\cO[D/D_i]   \ar[r] & \wt{N}\ar[r] & 0\end{tikzcd}$$
(see e.g. \cite[(3.6)]{GreenWalk}). 
But then for each $g\in D$, by Lemma~\ref{lem:MD/Di}(a) and the assumption that $\wt{N}$ has determinant $1$, we have
$$\det(g, \Omega_{D/D_i}(\wt{N})) \underbrace{\det(g,\wt{N})}_{=1} = \det(g,\cO[D/D_i]) =1 \,,$$
hence $\det(g, \Omega_{D/D_i}(\wt{N}))=1$, as \smallskip required.

\item[\rm(c)] The first claim follows from (a) since $k=M_1$. The second claim holds by (b). For the third claim, we consider again the $D_i$-relative projective cover of the trivial $\cO D$-lattice
 $$\begin{tikzcd}[cramped, sep=small]
 0\arrow[r] & \Omega_{D/D_i}(\cO) \ar[r] & \cO[D/D_i] \ar[r] &\cO\ar[r] & 0\,.
 \end{tikzcd}$$
Thus, computing in the Grothendieck ring of $KD$, we obtain that the $K$-character afforded by the lift of determinant $1$ of $\Omega_{D/D_i}(k)$ is 
$$\chi_{\Omega_{D/D_i}(k)}=\chi_{k[D/D_i]}-\chi_{k}=\big(\sum_{\substack{0\leq \kappa \leq  p^n -1\\ p^i | \kappa}}\lambda^D_\kappa \big) \ - \ \lambda^D_0= \sum_{\substack{1\leq \kappa \leq  p^n -1\\ p^i | \kappa}}\lambda^D_\kappa\,,$$
where the second equality holds by Lemma~\ref{lem:MD/Di}(b).
\end{enumerate}
\end{proof}

\begin{prop}\label{prop:charW_D}
Let $\mathcal{W}:=
\Omega_{D/D_{i(0)}}\circ\Omega_{D/D_{i(1)}}\circ\cdots\circ\Omega_{D/D_{i(s)}}(k)$ be an indecomposable capped endo-permutation $kD$-module, where $s\geq 0$ and  
$0 \leq i(0) < i(1) < \cdots <i(s) \leq n-1$ are integers and we set $s=-1$ if $\mathcal{W}=k$.
Then, in the Grothendieck ring of $KD$, the ordinary $K$-character afforded by the lift of determinant~1 of $\mathcal{W}$ is 
$$
    \chi_\mathcal{W}=\sum_{j=0}^s (-1)^j\Big(   \sum_{\substack{0\leq \kappa\leq p^n -1\\p^{i(j)}|\kappa}}\lambda_\kappa^{D}  
                                                                       \Big) \ + \ (-1)^{s+1}\lambda_0^{D} \,.
$$
\end{prop}

\begin{proof}
We proceed by induction on $s$.
If $s=-1$, then $\mathcal{W}=k=M_1=M_{D/D_n}$, hence  $\chi_\mathcal{W}=\lambda_0^D$  by Lemma~\ref{lem:MD/Di}. If $s=0$, then $\mathcal{W}=\Omega_{D/D_{i(0)}}(k)$ and  by Lemma~\ref{lem:OmegaD/Di} we have
$$\chi_\mathcal{W}=\sum_{\substack{1\leq \kappa \leq  p^n -1\\ p^{i(0)} | \kappa}}\lambda^D_\kappa = \big(\sum_{\substack{0\leq \kappa \leq  p^n -1\\ p^{i(0)} | \kappa}}\lambda^D_\kappa\big)-\lambda_0^D\,.$$
Hence the formula holds for $s=-1$ and $s=0$. So let us assume that $s\geq 1$ and  set 
$$\mathcal{W}':=\Omega_{D/D_{i(1)}}\circ\cdots\circ\Omega_{D/D_{i(s)}}(k)$$
$r(\mathcal{W}')=\dim_k(\mathcal{W}')$.  Because $i(0) < i(1) < \cdots <i(s) \leq n-1$, we have $1\leq r(\mathcal{W}')\leq p^{n-i(0)}-1$, hence 
$$\mathcal{W}=\Omega_{D/D_{i(0)}}(\mathcal{W}')=M_{|D/D_{i(0)}|-r(\mathcal{W}')}$$
by Lemma~\ref{lem:OmegaD/Di}(a). 
Now, by Lemma~\ref{lem:OmegaD/Di}(c) and (b) (applied inductively), we obtain that 
$$\Omega_{D/D_{i(1)}}\circ\cdots\circ\Omega_{D/D_{i(s)}}(\cO)=:\widetilde{\mathcal{W}'}$$
 is the unique lift of determinant 1 of $\mathcal{W}'$ and  by the induction hypothesis
$$\chi_{\mathcal{W}'}= \sum_{j=1}^s (-1)^{j+1}\Big(   \sum_{\substack{0\leq \kappa\leq p^n -1\\p^{i(j)}|\kappa}}\lambda_\kappa^{D}  \Big) \ + \ (-1)^{s}\lambda_0^{D}\,.$$
Again, by Lemma~\ref{lem:OmegaD/Di}(b), $\Omega_{D/D_{i(0)}}(\widetilde{\mathcal{W}'})$ is the unique lift of determinant 1 of $\Omega_{D/D_{i(0)}}(\mathcal{W}')=\mathcal{W}$.
Hence in the Grothendieck ring of $KD$, we have

\begin{equation*}
\begin{split}
  \chi_\mathcal{W} =\chi_{M_{|D/D_{i(0)}|}}-\chi_{\mathcal{W}'}  
             &=  \sum_{\substack{0\leq \kappa \leq  p^n -1\\ p^{i(0)} | \kappa}}  \lambda^D_\kappa   -  \Big(    \sum_{j=1}^s (-1)^j\Big(   \sum_{\substack{0\leq \kappa\leq p^n -1\\p^{i(j)}|\kappa}}\lambda_\kappa^{D}  \Big) \ + \ (-1)^{s}\lambda_0^{D} \Big) \\
               &=   \sum_{j=0}^s (-1)^j\Big(   \sum_{\substack{0\leq \kappa\leq p^n -1\\p^{i(j)}|\kappa}}\lambda_\kappa^{D}  \Big) \ + \ (-1)^{s+1}\lambda_0^{D}\,.
\end{split}
\end{equation*}
\end{proof}

\vspace{2mm}
\subsection{Characters of the Morita correspondents}

We can now proceed to describe the $K$-characters of the Morita correspondents of the trivial source $\bc$-modules under the character bijection $\Gamma_K$ of \S\ref{ssec:PuigForBlockc}. \\

Throuhout this subsection we fix a vertex $D_i\leq D$ with $1\leq i\leq n$ and we denote by  $\rho_{(i,W)}$ for the $K$-character afforded by the unique lift of determinant $1$ of the indecomposable capped endo-permutation $kD_i$-module $\Capp\circ\Res^D_{D_i} (W)$.

\begin{lem}\label{lem:charMorCor}
Let $M$ be the unique trivial source $\bc$-module with vertex $1<D_i\leq D$ and let $\psi_M$ denote the $K$-character afforded by its trivial source lift.   Then
$$\Gamma_{K}^{-1}(\psi_M)=\Ind_{D_i}^{D}( \rho_{(i,W)} )\,.$$
\end{lem} 

\begin{proof}
Follows directly from Lemma~\ref{lem:tsc} and  the definition of the character bijection $\Gamma_{K}$ of~\S\ref{ssec:PuigForBlockc}.
\end{proof}

\begin{lem}\label{lem:charWDi}
Let $1\leq i\leq n$ and let $\mathcal{W}:=\Omega_{D_i/D_{i(0)}}\circ\cdots\circ\Omega_{D_i/D_{i(t)}}(k)$, where $t\geq 0$ and $0 \leq i(0) < i(1) < \cdots <i(t) \leq i-1$ are integers.
Then: 
\begin{enumerate}
\item[\rm(a)] $\mathcal{W}=\sum_{j=0}^{t}(-1)^{j}\Ind_{D_{i(j)}}^{D_i}(k)+(-1)^{t+1}k$ in the Grothendieck ring of $kD_i$; and 
\item[\rm(b)]$\Ind_{D_i}^{D}(\mathcal{W})=\sum_{j=0}^{t}(-1)^{j} \Ind_{D_{i(j)}}^{D}(k)   +  (-1)^{t+1}   \Ind_{D_i}^{D}(k)$ in the Grothendieck ring of $kD$. 
\end{enumerate}
Moreover, in the Grothendieck ring of $KD$,
$$\Ind_{D_i}^{D}(\chi_\mathcal{W})=\sum_{j=0}^{t}(-1)^{j}
\Big(   \sum_{\substack{0\leq \kappa \leq  p^n -1\\ p^{i(j)} | \kappa}}  \lambda^D_\kappa   
\Big) 
+  (-1)^{t+1} \Big(  \sum_{\substack{0\leq \kappa \leq  p^n -1\\ p^{i} | \kappa}}   \lambda^D_\kappa 
                       \Big) \,.$$
\end{lem}

\begin{proof}First we note that $\mathcal{W}$ is an indecomposable capped endo-permutation $kD_i$-module. 
\begin{enumerate}
\item[\rm(a)] We proceed by induction on $t$. If $t=0$, then considering a $D_{i(0)}$-relative projective cover of the trivial module
$$\begin{tikzcd}[cramped, sep=small] 0\arrow[r] & \Omega_{D/D_i}(k) \ar[r] & \Ind_{D_{i(0)}}^{D_i}(k) \ar[r] & k\ar[r] & 0\end{tikzcd}$$
yields $\mathcal{W}= \Ind_{D_{i(0)}}^{D_i}(k)-k$, as required. 
Now, given $t>1$, we may decompose 
\begin{equation*}
\begin{split}
             \mathcal{W} &=    \Omega_{D_i/D_{i(0)}} \Big[  \Omega_{D_i/D_{i(1)}}  \circ\cdots\circ\Omega_{D_i/D_{i(t)}}(k)         \Big]  \\
                 &=    P_{D_i/D_{i(0)}}\Big[   \Omega_{D_i/D_{i(1)}}  \circ\cdots\circ\Omega_{D_i/D_{i(t)}}(k)         \Big] -  \Big[  \Omega_{D_i/D_{i(1)}}  \circ\cdots\circ\Omega_{D_i/D_{i(t)}}(k)         \Big]  \,.
\end{split}
\end{equation*}
As  $\dim_k\left(  \Omega_{D_i/D_{i(1)}}  \circ\cdots\circ\Omega_{D_i/D_{i(t)}}(k)    \right)< |D_i/D_{i(0)}|$, we have 
$$P_{D_i/D_{i(0)}}\big[   \Omega_{D_i/D_{i(1)}}  \circ\cdots\circ\Omega_{D_i/D_{i(t)}}(k)   \big]=P_{D_i/D_{i(0)}}(k)$$ and the 
the induction hypothesis yields
\begin{equation*}
\begin{split}
             \mathcal{W} &=      P_{D_i/D_{i(0)}}(k) - \Big[\sum_{j=1}^{t}(-1)^{j+1}\Ind_{D_{i(j)}}^{D_i}(k)+(-1)^{t}k\Big]\\
                 &=    \sum_{j=0}^{t}(-1)^{j}\Ind_{D_{i(j)}}^{D_i}(k)+(-1)^{t+1}k
\end{split}
\end{equation*}
\item[\rm(b)]  It follows from (a) that
\begin{equation*}
\begin{split}
              \Ind_{D_i}^{D}(\mathcal{W}) & =          \Ind_{D_i}^{D}\!\Big(   \sum_{j=0}^{t}(-1)^{j}\Ind_{D_{i(j)}}^{D_i}(k)+(-1)^{t+1}k     \Big)\\
                                           & =            \sum_{j=0}^{t}(-1)^{j}   \left(  \Ind_{D_i}^{D}\circ \Ind_{D_{i(j)}}^{D_i}(k) \right) + (-1)^{t+1}\Ind_{D_i}^{D}(k)\\
                                           &=              \sum_{j=0}^{t}(-1)^{j} \Ind_{D_{i(j)}}^{D}(k)   +  (-1)^{t+1}   \Ind_{D_i}^{D}(k)
\end{split}
\end{equation*}
\end{enumerate}The last claim is now a direct consequence of Lemma~\ref{lem:MD/Di}(b).
\end{proof}

\begin{prop}\label{prop:charMoritaCorresp}
Let $W=W(0<i_0<i_1<\ldots< i_s<n)$ be the indecomposable capped endo-permutation module parametrising the source algebra of the block $\bB$. Let $M$ be the unique trivial source $\bc$-module with vertex $D_i$ and  let $\psi_M$ denote the $K$-character afforded by its trivial source lift to $\cO$.
Then 
$$\Gamma_{K}^{-1}(\psi_M)= \sum_{j=0}^{t(i)}(-1)^{j}
\Big(   \sum_{\substack{0\leq \kappa \leq  p^n -1\\ p^{i_j} | \kappa}}  \lambda^D_\kappa   
\Big) +  (-1)^{t(i)+1}
\Big(  \sum_{\substack{0\leq \kappa \leq  p^n -1\\ p^{i} | \kappa}}   \lambda^D_\kappa 
\Big)\,, $$
where  $t(i):=\text{max}\{ 0\leq j\leq s\mid i_j\leq i-1\}$ if $W\cong k$ and $t(i):=-1$ if $W=k$.  
\end{prop}

\begin{proof}
By Lemma~\ref{lem:charMorCor}, $\Gamma_{K}^{-1}(\psi_M)=\Ind_{D_i}^{D}( \rho_{(i,W)})$, where $\rho_{(i,W)}$ is the  $K$-character of the  lift of determinant $1$ of the indecomposable capped endo-permutation $kD_i$-module $\Capp\circ\Res^D_{D_i} (W)$. By \cite[\S4.5]{HL19}, we have 
\begin{equation*}
\begin{split}
      \Capp\circ \Res^{D}_{D_i}\left(W\right)  & = \Omega_{D_{i}/D_0}^{a_0}\circ\Omega_{D_i/D_1}^{a_1}\circ\cdots\circ\Omega_{D_i/D_{i-1}}^{a_{i-1}}(k)\\
                                                                     & =   \Omega_{D_i/D_{i_0}}\circ\cdots\circ\Omega_{D_i/D_{i_t}}(k)   \,,    
\end{split}
\end{equation*}
where  $t:=t(i)$. Therefore it follows from Lemma~\ref{lem:charWDi}(b) that 
$$\Gamma_{K}^{-1}(\psi_M)=\Ind_{D_i}^{D}( \rho_{(i,W)})=\sum_{j=0}^{t(i)}(-1)^{j}\Big(   \sum_{\substack{0\leq \kappa \leq  p^n -1\\ p^{i_j} | \kappa}}  \lambda^D_\kappa   \Big) +  (-1)^{t(i)+1} \Big( \sum_{\substack{0\leq \kappa \leq  p^n -1\\ p^{i} | \kappa}}   \lambda^D_\kappa \Big)  \,.$$
\end{proof}


\section{Step 2: Characters of the trivial source $\bb$-modules}\label{sec:charN1}

Throughout this section, we assume $W=W(0<i_0<i_1<\ldots< i_s<n)$ according to Notation~\ref{nota:W} is the indecomposable capped endo-permutation module parametrising the source algebra of the block $\bB$. We let $1<D_i\leq D$ ($1\leq i\leq n$) be a fixed vertex and we set 
$t(i):=\text{max}\{ 0\leq j\leq s\mid i_j\leq i-1\}$ if $W\ncong k$  and $t(i):=-1$ when $W=k$.\\

First we recover the characters of the trivial source $\bc$-modules.

\begin{lem}\label{lem:charC1}
Let $M$ be the unique trivial source $\bc$-module with vertex $D_i$ and let $\psi_M$ denote the $K$-character afforded by its trivial source lift to $\cO$.   Then
\begin{equation*}
\begin{split}
 \psi_M   & = \sum_{j=0}^{t(i)}(-1)^{j}\Big(   \sum_{\substack{0\leq \kappa \leq  p^n -1\\ p^{i_j} | \kappa}} \psi_{ \lambda^D_\kappa }   \Big) +  (-1)^{t(i)+1}\Big(  \sum_{\substack{0\leq \kappa \leq  p^n -1\\ p^{i} | \kappa}} \psi_{  \lambda^D_\kappa  } \Big)\\
              & =  \sum_{j=0}^{t(i)}(-1)^{j}\Big(   \sum_{\substack{1\leq \kappa \leq  p^n -1\\ p^{i_j} | \kappa}} \psi_{ \lambda^D_\kappa }   \Big) +  (-1)^{t(i)+1}\Big(  \sum_{\substack{1\leq \kappa \leq  p^n -1\\ p^{i} | \kappa}} \psi_{  \lambda^D_\kappa  } \Big)     + d(W,D_i)  \psi_{ \lambda^D_0 }      \,,         
\end{split}
\end{equation*}
where $d(W,D_i):=0$ if  $t(i)$ is even and   $d(W,D_i):=1$ if  $t(i)$ is odd. 
\end{lem} 

\begin{proof}
Applying $\Gamma_K$ to the formula in Proposition~\ref{prop:charMoritaCorresp} yields the first equality. The second equality is straightforward, indeed,  we only write the unique non-exceptional character $\psi_{ \lambda^D_0 }$ in a separate summand. 
\end{proof}

Next we need to induce the above characters in turn to the stabiliser $T(\bc)$ and then $N_1$ in order to compute the $K$-characters of the trivial source $\bb'$-modules and  $\bb$-modules. 

\begin{rem}\label{rem:labellingIrrEx}
Recall that we write $\Irr_K(\bc)=\{\psi_{\lambda^D_\kappa}   \mid 1\leq \kappa\leq p^{n}-1\}$. Then the following assertions follow from Clifford-theoretic arguments (see \cite[\S19]{AlperinBook}): 
\begin{enumerate}
\item[\rm(1)] For $\psi_{\lambda^D_0}$, the unique non-exceptional character of $\bc$, we have 
$$\Ind_{C_G(D_1)}^{T(\bc)}(\psi_{\lambda^D_0})=\widetilde{\psi}_1+\ldots+\widetilde{\psi}_e\,,$$ 
where $\{\widetilde{\psi}_1,\ldots,\widetilde{\psi}_e\}=\Irr'(\bb')$  (each $\widetilde{\psi}_j$ extends $\psi_{\lambda^D_0}$);
\item[\rm(2)] $\Ind_{C_G(D_1)}^{T(\bc)}(\psi_{\lambda^D_\kappa})=:\widetilde{\psi}_{\lambda^D_\kappa}\in\Irr_{\text{Ex}}(\bb')$ for each exceptional character $\psi_{\lambda^D_\kappa}\Irr_{\text{Ex}}(\bc)$; 
\item[\rm(3)] $\Irr'(\bb)=\{\theta_1,\ldots,\theta_e\}$ where $\theta_{j}:=\Ind_{T(\bc)}^{N_1}(\widetilde{\psi}_j)$ for each $1\leq j\leq e$ and  
$$\Ind_{T(\bc)}^{N_1}(\widetilde{\psi}_{\lambda^D_\kappa})=:\theta_{\lambda^D_\kappa}\in\Irr_{\text{Ex}}(\bb)\quad \text{ for each }1\leq \kappa\leq p^{n}-1$$
 as the  theorem of Fong-Reynolds gives a source-algebra equivalence between $\bb'$ and $\bb$ induced by induction from $T(\bc)$ to $N_1$. 
 (See \cite[1.5.Theorem]{KKW04}.)

 \item[\rm(4)] Let $E$ be the inertial quotient of $\bB$. This is a cyclic subgroup of order $e$ of $N_G(D)/C_G(D)$, hence acts by inner automorphisms on $D=\langle u\rangle$ and embeds as a subgroup of $\Aut(D)\cong (\IZ/p^n\IZ)^{\times}$.
Hence, writing $E=\langle \bar{h}\rangle$ with $h\in N_G(D)$, there exists $\bar{a}\in (\IZ/p^n\IZ)^{\times}$  of order $e$ such that 
$$h^{-1}uh=u^{a}\,,$$
where $0\leq a< p^n$ is coprime to $p$ since $e\mid p-1$, so that the group  $E$ acts by conjugation on $\Irr_{\text{Ex}}(D)$ via
$$(\lambda_{\kappa}^D)^{\bar{h}}(u)=\lambda_\kappa^D(h^{-1}uh)=\lambda_\kappa^D(u^a)=\zeta^{\kappa a}=\lambda_{\kappa a}^D(u)\,.$$
Hence � $(\lambda_{\kappa}^D)^{\bar{h}^\alpha}=\lambda_{\kappa a^{\alpha}}^D$ and each orbit has length $e$.  Therefore, fixing a set of representatives of the orbits of this action, say $\{\lambda_{\kappa(r)}\mid 1\leq r\leq m\}=:\Lambda$  (where $m$ is the exceptional multiplicity of $\bB$),  we may rewrite
$$\Irr_{\text{Ex}}(D)=\bigsqcup_{r=1}^{m}\{\lambda_{\kappa(r)a^\alpha}^D\mid 0\leq \alpha\leq e-1\}\,,$$
where $\lambda_{\kappa(r,\alpha)}^D:D\lra K^{\times}, u\mapsto \zeta^{\kappa(r)a^{\alpha}}$.   
It follows that 
$$\Ind_{C_G(D_1)}^{T(\bc)}( \psi_{   \lambda_{\kappa(r)a^{\alpha}}^D })=\Ind_{C_G(D_1)}^{T(\bc)}( \psi_{  \lambda_{\kappa(r')a^{\alpha'}}^D}  )\quad\Longleftrightarrow\quad r=r'\,,$$
thus we may set 
$$\theta_{\lambda_{\kappa(r)}}:=\Ind_{C_G(D_1)}^{N_1}(\psi_{\lambda_{\kappa(r)a^{\alpha}}^D})\quad\text{for each }1\leq r\leq m, 0\leq\alpha\leq e-1\,,$$
so that by the above  $\Irr_{\text{Ex}}(\bb)=\{\theta_{\lambda_{\kappa(r)}}\mid 1\leq r\leq m\}=\{\theta_{\lambda}\mid \lambda\in\Lambda\}$\,.
\end{enumerate}
\end{rem}
\bigskip

\begin{cor}\label{cor:tsb}
Let $1<D_i\leq D$  be as above. Let  $Y_1,\ldots, Y_e$ be the $e$ pairwise non-isomorphic trivial source $\bb$-modules with vertex $D_i$. 
For each $1\leq x\leq e$ let $\chi_{Y_x}=\Psi_{Y_x}+\Xi_{Y_x}$ be the $K$-character afforded by  the trivial source lift of $Y_x$ to $\cO$ (see Notation~\ref{nota:chiM}). Then  the following assertions hold:
\begin{enumerate}
\item[\rm(a)]  if $t(i)$ is odd, then 
without loss of generality 
we may assume that we have chosen the labelling such that  $\Psi_{Y_x}=\theta_x$, whereas $\Psi_{Y_x}=0$  if  $t(i)$ is even; and
\item[\rm(b)]  $$\Xi_{Y_x}=      \sum_{j=0}^{t(i)}(-1)^{j} \Big(  \sum_{\substack{1\leq r \leq m\\ p^{i_j} | \kappa(r)}}  \theta_{\lambda_{\kappa(r)}}  \Big)   + (-1)^{t(i)+1}\Big(\sum_{\substack{1\leq r \leq m\\ p^{i} | \kappa(r)}}   \theta_{\lambda_{\kappa(r)}}  \Big) \,.$$
\end{enumerate}
\end{cor}

\begin{proof}
First assume that $Y_1,\ldots, Y_e$ are hooks. Then by \cite[Corollary 5.2(c)]{HL19}, we must have $W=k$ and $D_i=D_n$. In addition, as $\sigma(\bb)$ is a star with $e$ edges and exceptional vertex at its centre, $Y_1,\ldots, Y_e$ are simple.
Hence we may assume that we have chosen the labelling such that $\chi_{Y_x}=\theta_x=\Psi_{Y_x}$ and $\Xi_{Y_x}=0$ for each $1\leq x\leq e$. Hence (a) and (b) hold in this case. 
\par
We may now assume that $Y_1,\ldots, Y_e$ are not hooks. If $M$ denotes the unique trivial source $\bc$-module with vertex $D_i$, then by Clifford theory
$$\Ind_{C_G(D_1)}^{T(\bc)}(M)=M_1\oplus \cdots\oplus M_e$$
is the direct sum of the $e$ pairwise non-isomorphic trivial source $\bb'$-modules with vertex $D_i$ and 
$$\Ind_{C_G(D_1)}^{N_1}(M)=Y_1\oplus \cdots\oplus Y_e \quad\text{ with } Y_j=\Ind_{T(\bc)}^{N_1}(M_j)\;\forall\;1\leq j\leq e\text{ (w.l.o.g.)}$$
is the direct sum of the $e$ pairwise non-isomorphic trivial source $\bb$-modules with vertex $D_i$. 
At the level of $K$-characters, we obtain from Lemma~\ref{lem:charC1} and Remark~\ref{rem:labellingIrrEx} that 
\begin{equation*}
\begin{split}
  \Ind_{C_G(D_1)}^{N_1}(\psi_M) & =    \sum_{j=0}^{t(i)}(-1)^{j}\Big(   \sum_{\substack{1\leq \kappa \leq  p^n -1\\ p^{i_j} | \kappa}} \Ind_{C_G(D_1)}^{N_1}(\psi_{ \lambda^D_\kappa })   \Big) +  (-1)^{t(i)+1}\Big(  \sum_{\substack{1\leq \kappa \leq  p^n -1\\ p^{i} | \kappa}} \Ind_{C_G(D_1)}^{N_1}(\psi_{  \lambda^D_\kappa}) \!\Big)  \\
                                                        &\phantom{= }    + d(W,D_i)  \Ind_{C_G(D_1)}^{N_1}(\psi_{ \lambda^D_0 })     \\
                                                        & =     \sum_{j=0}^{t(i)}(-1)^{j}e\Big( \!   \sum_{\substack{1\leq r \leq m\\ p^{i_j} | \kappa(r)}}  \theta_{\lambda_{\kappa(r)}} \!  \Big)   \! +\!  (-1)^{t(i)+1}e\Big(\! \sum_{\substack{1\leq r \leq m\\ p^{i} | \kappa(r)}}   \theta_{\lambda_{\kappa(r)}} \! \Big)  \! +\! d(W,D_i)(\theta_1+\ldots+\theta_e)  \,.               
\end{split}
\end{equation*}
As by Lemma~\ref{lem:=exc} we have  $\Xi_{Y_1}=\ldots=\Xi_{Y_e}$ and the multiplicity of each irreducible constituent of this character is one, we have 
$$\Xi_{Y_x}=      \sum_{j=0}^{t(i)}(-1)^{j} \Big(  \sum_{\substack{1\leq r \leq m\\ p^{i_j} | \kappa(r)}}  \theta_{\lambda_{\kappa(r)}}  \Big)   + (-1)^{t(i)+1}\Big(\sum_{\substack{1\leq r \leq m\\ p^{i} | \kappa(r)}}   \theta_{\lambda_{\kappa(r)}}  \Big) $$
for each $1\leq x\leq e$.
\end{proof}

\begin{rem}\label{rem:pathsb}
According to Janusz' classification of the indecomposable modules in blocks with cyclic defect groups \cite{Jan69} 
a non-simple trivial source $\bb$-module $Y_x$ ($1\leq x\leq e$)  as in Corollary~\ref{cor:tsb}  can only correspond to paths on the Brauer tree $\sigma(\bb)$ of the form
$$\xymatrix@R=0.0000pt@C=30pt{	
		{_{\theta_x}}&{_{\theta_{\Lambda}}}\\
		{\Circle } \ar@<0.3ex>[r]^{S_x}  &{\CIRCLE}\ar@<0.3ex>[l]^{S_x}  
}$$
or of the form
$$\xymatrix@R=0.0000pt@C=30pt{
                {_{\theta_{x_1}}}& \\	
		{\Circle }\ar[ddr]^{S_{x_1}} &  \\
		&{_{\chi^{}_{\Lambda}}} \\
		&{\CIRCLE}\ar[dddl]^{\:\:S_{x_2}}\\
		&\\
		&\\
		{\Circle }&  \\
		{^{\theta_{x_2}}}& 
}$$
because  $\sigma(\bb)$ is a star with exceptional vertex at its center. Therefore, if $e>1$,  it is a priori clear that any lift of $Y_x$ affords a $K$-character of the form $d_x\theta_x+\theta_{\Lambda'}$ for some $d_x\in\{0,1\}$ and some $\Lambda'\subseteq\Lambda$. See \cite[Theorem A.1]{HL19}. \par
Now, if $e>1$, then a trivial source $\bb$-module $Y_x$ with $\Psi_{Y_x}=\theta_x$ corresponds to a path of the first type and if $\Psi_{Y_x}=0$, then $Y_x$ corresponds to a path of the second type. See \cite[Theorem A.1]{HL19}. If $e=1$ only the first type of paths exist. In this case Corollary~\ref{cor:tsb} tells us whether $\theta_x$ occurs as a constituent in $\chi_{Y_x}$ or not. 
\end{rem}


\vspace{6mm}
\section{Step 3: From $\bb$ to $\bB$, the exceptional constituents}\label{sec:levelG}

For the passage from $\bb$ to $\bB$, we first need to describe the labelling of the exceptional $K$-characters of $\bB$ which  we will use in the sequel. Recall that  we write $\Irr'(\bb)=\{\theta_1,\ldots,\theta_e\}$ and $\Irr_{\text{Ex}}(\bb)=\{\theta_{\lambda}\mid \lambda\in\Lambda\}$, where 
$$\Lambda=\{\lambda_{\kappa(r)}\mid 1\leq r\leq m\}$$
is defined in Remark~\ref{rem:labellingIrrEx}. Moreover, we write $\Irr'(\bB)=\{\chi_1,\ldots,\chi_e\}$, where we may assume that for each $1\leq x\leq e$, $\chi_x$ is the $K$-character of the Green correspondent in $\bB$ of the simple $\bb$-module $S_x$ affording the $K$-character~$\theta_x$. Then the standard  labelling of the exceptional characters of $\bB$ is achieved as follows: 
if $\Delta:\IZ\Irr(\bb)\lra\IZ\Irr(\bB)$  denotes the homomorphism of abelian groups  induced by the functor $1_{\tilde\bB}\cdot\Ind_{N_1}^G$, there exists a sign $\delta\in\{\pm1\}$ and $\{\chi_{\lambda}\mid \lambda\in\Lambda\}$ such that for all pairs $\lambda,\lambda'\in\Lambda$, we have 
$$\Delta(\theta_{\lambda}-\theta_{\lambda'})=\delta(\chi_{\lambda}-\chi_{\lambda'})\,.$$
By \cite[Theorems 11.10.2(ii)]{LinckBook} this yields the existence of  a perfect isometry 
$$\mathcal{I}:\IZ\Irr(\bb)\lra\IZ\Irr(\bB)$$ 
sending each $\theta_x\in\Irr'(\bb)$  to $\mathcal{I}(\theta_x)=\delta(\theta_x)\chi_x$ with $\delta(\theta_x)\in\{\pm1\}$ and each $\theta_{\lambda}\in\Irr_{\text{Ex}}(\bb)$ to 
$\mathcal{I}(\theta_{\lambda})=\delta\chi_{\lambda}$ with $\delta\in\{\pm1\}$ independent of $\lambda\in\Lambda$. 

\begin{rem}\label{rem:RickardComplex}
By results of Rickard and Rouquier, see \cite[Theorem 11.12.1]{LinckBook}, there is a 2-term splendid Rickard complex
$$
M^\bullet : \ 0\rightarrow N\rightarrow M\rightarrow 0
$$
of $(\bB,\bb)$-bimodules, where $N$ and $M$  are in degrees $-1$ and $0$ respectively, $M:=1_{\bB}{\cdot}kG{\cdot}1_{\bb}$, and $N$ is a certain direct summand of the projective cover
of $M$ as $(\bB,\bb)$-bimodule. Thus, by \cite[Corollary 9.3.3]{LinckBook}, the complex $M^\bullet$ induces 
another perfect isometry 
$$I: \mathbb Z{\mathrm{Irr}}(\bb)\rightarrow\mathbb Z{\mathrm{Irr}}(\bB)$$ such that on the one hand 
for each $\theta\in{\mathrm{Irr}}(\bb)$, we have  $I(\theta)=\varepsilon (\theta)\chi$ for a certain $\chi\in{\mathrm{Irr}}(\bB)$ and a sign $\varepsilon(\theta)\in\{\pm 1\}$, and on the other hand
\begin{equation}\label{PerfIso}
 I(\theta)= (\chi_M - \chi_N)\otimes_{K\bb}\theta 
\end{equation}
for every $\theta\in\mathbb Z{\mathrm{Irr}}(\bb)$. Moreover, because $I$ and $\mathcal{I}$ are two perfect isometries, in fact it follows from \cite[Theorems 11.1.12 and 11.10.2(ii)]{LinckBook}
that
$I$ sends the non-exceptional characters $\theta_x\in\Irr'(\bb)$  to $I(\theta_x)=\varepsilon(\theta_x)\chi_x$ for each $1\leq x\leq e$ and the exceptional characters $\theta_{\lambda}\in\Irr_{\text{Ex}}(\bb)$ to 
$$I(\theta_{\lambda})=\varepsilon\cdot \chi_{\lambda}$$
where  $\varepsilon := \varepsilon (\theta_{\lambda(1)}) = \ldots = \varepsilon (\theta_{\lambda(m)})$.
\end{rem}

\begin{lem}\label{lem:LinearlyIndep}
Let $\chi$ be a $K$-character of $G$ afforded by an $\mathcal OG$-lattice which is a lift of an indecomposable 
$\bB$-module~$X$. 
Furthermore, suppose that there exist a subset $\Lambda'$ of $\Lambda$, a sign $\varepsilon\in\{\pm 1\}$ and integers $\alpha_1, \cdots, \alpha_e, \beta\in \IZ$
such that 
$$
\chi = \sum_{x=1}^e\,\alpha_x\chi_x + \beta\chi_\Lambda+\varepsilon\chi^{}_{\Lambda'}.
$$
Then, either
$$ \chi = \sum_{x=1}^e\,\alpha_x\chi_x+ \chi^{}_{\Lambda'} \qquad \text{ or } \qquad \chi=\sum_{x=1}^e\,\alpha_x\chi_x + \chi^{}_{\Lambda\setminus\Lambda'}\,.$$
\end{lem}

(See Notation~\ref{nota:chiM}.)

\begin{proof}
We have 
$$
\chi=\sum_{x=1}^e\alpha_x\chi_x + \beta\chi_\Lambda+\varepsilon\chi^{}_{\Lambda'}
         =\sum_{x=1}^e\alpha_x\chi_x + \sum_{\lambda\in{\Lambda\setminus\Lambda'}}\beta\chi_{\lambda}
                                                      + \sum_{\lambda\in\Lambda'}(\beta+\varepsilon)\chi_{\lambda}.
$$
Since $(\Lambda\setminus\Lambda')\cap\Lambda'=\emptyset$ and $\langle\chi, \chi_{\lambda}\rangle^G\in\{0,1\}$ for each $\lambda\in\Lambda$, we have that $\beta,\beta+\varepsilon\in\{0,1\}$ (see Notation~\ref{nota:chiM}).  Hence $\beta=1-\varepsilon$.
Therefore, $\varepsilon=1$ yields $\chi = \sum_{x=1}^e\,\alpha_x\chi_x + \chi^{}_{\Lambda'}$, whereas   $\varepsilon=-1$ yields $ \chi=\sum_{x=1}^e\,\alpha_x\chi_x + \chi^{}_{\Lambda\setminus\Lambda'}$.
\end{proof}

\begin{prop}\label{prop:PerfIso}
Let $Y$ be {a} non-projective trivial source $\bb$-module and let ${X:=f^{-1}(Y)}$ be its Green correspondent in $\bB$. 
Write $\Xi_Y= \theta_{\Lambda'}$ with $\Lambda'\subseteq\Lambda$ for the exceptional part of~$\chi_Y$.
Then the exceptional part of~$\chi_X$ is
$$\Xi_{X}= \chi^{}_{\Lambda'}\quad\text{ or }\quad\Xi_{X} = \chi^{}_{\Lambda\setminus\Lambda'}\,.$$
\end{prop}

\begin{proof}
According to Remark~\ref{rem:pathsb}, we may write $\Psi_Y=d_0\theta_{x_0}$ for some $1\leq x_0\leq e$ and some $d_0\in\{0,1\}$, so that $\chi_Y= d_0\theta_{x_0}+\theta_{\Lambda'}$.  Then, it follows from Remark~\ref{rem:RickardComplex} that\\
\begin{equation*}
\begin{split}
  (\chi_M-\chi_N)\otimes_{K\bb}\chi_Y  = I(\chi_Y)   & = I(d_0\theta_{x_0}+\theta_{\Lambda'})   \\
        & =       I\Big(d_0\theta_{x_0}+\sum_{\lambda\in\Lambda'}\theta_{\lambda'}\Big)     \\
        & = \varepsilon(\theta_{x_0})d_0\chi_{x^{}_0} + \sum_{\lambda\in\Lambda'} \varepsilon\chi_{\lambda'}\\
        &=  \varepsilon(\theta_{x_0})d_0\chi_{x^{}_0} + \varepsilon\chi^{}_{\Lambda'}
\end{split}
\end{equation*}\\
Now, on the one hand, as $M$ induces a stable equivalence of Morita type between $\bb$ and $\bB$, we have
$$
M\otimes_{\bb}Y = X\oplus \text{(projective }\bB\text{-module)}.
$$
Thus $\chi_{M\otimes_{\bb} Y} = \chi_X + \Phi\,,$ where $\Phi$ is the  character a  projective $\bB$-module. By Lemma~\ref{lem:PIMhooksSimples} we can write 
$$\Phi = \sum_{x=1}^{e}\alpha_x\chi_x+\alpha\chi_\Lambda$$
for non-negative integers  $\alpha_1,\ldots,\alpha_e,\alpha\in\IZ_{\geq0}$. On the other hand, $N$ is projective as a $(\bB,\bb)$-bimodule, hence $N\otimes_{\bb}Y$ is a projective left $\bB$-module. Thus again by Lemma~\ref{lem:PIMhooksSimples} we can write
$$\chi_N = \sum_{x=1}^{e}\beta_x\chi_x+\beta\chi_\Lambda$$
for non-negative integers  $\beta_1,\ldots,\beta_e,\beta\in\IZ_{\geq0}$. It follows that
\begin{equation*}\label{character}
(\chi_M-\chi_N)\otimes_{K\bb}\,\chi_Y
=(\chi_M\otimes_{K\bb}\,\chi_Y) - (\chi_N\otimes_{K\bb}\,\chi_Y)
= \chi_X + \sum_{x=1}^e\gamma_x\chi_x + 
 (\alpha-\beta)\chi_\Lambda
\end{equation*}
for integers $\gamma_1,\ldots,\gamma_e\in\IZ$. Hence
$$
 \chi_X + \sum_{x=1}^e\gamma_x\chi_x + (\alpha-\beta)\chi_\Lambda 
=\varepsilon(\theta_{x_0})d_0\chi_{x^{}_0} +\varepsilon\chi^{}_{\Lambda'}
$$
so that 
$$
\chi_X =-\varepsilon(\theta_{x_0})d_0\chi_{x^{}_0}+ \sum_{x=1}^n\,(-\gamma_x)\chi_x + (\beta-\alpha)\chi_\Lambda + \varepsilon\chi^{}_{\Lambda'}
$$
and the claim follows from Lemma~\ref{lem:LinearlyIndep}.
\end{proof}

\noindent Next, we count the number of exceptional constituents of the trivial source $\bb$-modules.

 \begin{lem}\label{lem:cardLambda'}
Let $Y$ be a non-projective trivial source $\bb$-module with vertex $D_i$ ($1\leq i\leq n$). 
Write $\Psi_Y=d_0\theta_{x_0}$ for some $1\leq x_0\leq e$ and some $d_0\in\{0,1\}$ for the non-exceptional part of $\chi_Y$ and 
$\Xi_Y= \theta_{\Lambda'}$ with $\Lambda'\subseteq\Lambda$ for the exceptional part of~$\chi_Y$. Then
$$|\Lambda'|=\frac{\ell_i\cdot p^{n-i}-d_0}{e}\qquad\text{ and }\qquad |\Lambda\setminus\Lambda'|=m-\frac{\ell_i\cdot p^{n-i}-d_0}{e}\,.$$
\end{lem}

\begin{proof}
On the one hand, because the multiplicity of each irreducible constituent of $\Xi_Y$ is one, we have that
$$|\Lambda'|=\langle \Xi_Y, \Xi_Y \rangle_G =  \langle \chi_Y, \chi_Y \rangle_G-d_0\,.$$
Now, reduction modulo $p$ of $\theta_{x_0}$ yields  one simple constituent of $Y$ and for each $\lambda\in\Lambda'$ reduction modulo $p$ of $\theta_{\lambda}$ yields $e$ simple constituents of $Y$, hence reduction modulo $p$ of $\chi_Y=d_0\theta_{x_0}+\theta_{\Lambda'}$ yields
$$\ell(Y)=d_0+e|\Lambda'|\,$$
as $\bb$ is uniserial.   On the other hand, as trivial source $\bb$-modules and trivial source $\bc$-modules with vertex $D_i$ have the same length (see \cite[Corollary~4.5]{HL19}) and $\bc$ is Morita equivalent to $kD$, it follows from Lemma~\ref{lem:tsc} that  the length of $Y$ is 
$$\ell(Y)=\ell(U_{D_i}(W))=\dim_k(U_{D_i}(W))\,.$$
Therefore 
$$|\Lambda'|=\frac{\dim_k(U_{D_i}(W))-d_0}{e} \qquad\text{ and }\qquad |\Lambda\setminus\Lambda'|=m-\frac{\dim_k(U_{D_i}(W))-d_0}{e}$$
and the claim follows from the fact that $\dim_k(U_{D_i}(W))=\ell_i\cdot p^{n-i}$. 
 \end{proof}

\vspace{8mm}
\section{Step 4: Characters of the trivial source modules at the level of $G$}\label{sec:levelG}

 \begin{thm}\label{thm:main}
 Let $\bB$ be a block with non-trivial cyclic defect group $D$, inertial index $e$, and  exceptional multiplicity $m>1$. Let $W=W(0<i_0<i_1<\ldots< i_s<n)$ be the endo-permutation $kD$-module parametrising the source algebra of $\bB$. 
 Let $X$ be a trivial source $\bB$-module with vertex $D_i$ ($1\leq i\leq n$). 
 Set 
 $$\Xi(W,i):=      \sum_{j=0}^{t(i)}(-1)^{j} \Big(  \sum_{\substack{1\leq r \leq m\\ p^{i_j} | \kappa(r)}}  \chi_{\lambda_{\kappa(r)}}  \Big)   + (-1)^{t(i)+1}\Big(\sum_{\substack{1\leq r \leq m\\ p^{i} | \kappa(r)}}   \chi_{\lambda_{\kappa(r)}}  \Big)$$
 and 
  $$\overline{\Xi(W,i)}:=      \sum_{j=0}^{t(i)}(-1)^{j} \Big(  \sum_{\substack{1\leq r \leq m\\ p^{i_j} \nmid \kappa(r)}}  \chi_{\lambda_{\kappa(r)}}  \Big)   + (-1)^{t(i)+1}\Big(\sum_{\substack{1\leq r \leq m\\ p^{i} \nmid \kappa(r)}}   \chi_{\lambda_{\kappa(r)}}  \Big) \,,$$
where $t(i):=\text{max}\{ 0\leq j\leq s\mid i_j\leq i-1\}$ if $W\ncong k$  and $t(i):=-1$ \smallskip if~$W=k$.
\begin{enumerate}
\item[\rm(a)] If $e=1$ and  the Brauer tree of $\bB$ is
$
\begin{tikzcd}
\sigma(\bB) =  \overset{\chi_1}{{\Circle }}  \arrow[r, dash,"S_1"]  & \overset{\chi^{}_{\Lambda}}{{\CIRCLE}}
\end{tikzcd}\!,
$ then the following assertions hold:
\begin{itemize}
  \item[\rm(i)]  $\chi_X=d_0\chi_1+\Xi(W,i)$  in case    $\chi_1>0$, and 
  \item[\rm(ii)] $\chi_X=(1-d_0)\chi_1+\overline{\Xi(W,i)}$ in case  $\chi_1<0$,
\end{itemize}
where $d_0=1$ if $t(i)$ is odd and $d_0=0$ if $t(i)$  is \medskip even.
\item[\rm(b)] If $e>1$, then the following assertions hold.
\begin{itemize}
  \item[\rm\bf(1)] 
  If the vertex is $D_i=D$ and  $W=k$, then $X$ is a hook and there exists  $\chi\in\Irr^{\circ}(\bB)$ such that $\chi>0$ and   \smallskip $\chi_X=\chi$. 
  \item[\rm\bf(2)]  
  If $X$ corresponds to the  path
  $$
\xymatrix@R=0.0000pt@C=30pt{	
		{_{\chi_{x^{}_0}}}&{_{\chi_{x^{}_1}}}&{_{\chi_{x^{}_l}}}&{_{\chi^{}_{\Lambda}}}\\
		{\Circle } \ar@<0.3ex>[r]^{E_1}  &{\Circle }\ar@<0.3ex>[l]^{E_s}  \ar@{.}[r]    &{\Circle } \ar@<0.3ex>[r]^{E_{l+1}} & {\CIRCLE} \ar@<0.3ex>[l]^{E_{l+2}}
}
$$
where the  direction is  $\varepsilon=(1,-1)$,  $l\geq 0$, and  $\chi_{x^{}_0}$ is a leaf of $\sigma(\bB)$, \smallskip then:
\begin{itemize}
  \item[\rm(i)] $\chi_X=\sum_{z=0}^{l}\chi_z +  \overline{\Xi(W,i)}$ in case $l$ is odd, $\chi_{x^{}_0}>0$, $e\mid(\ell_i\cdot p^{n-i}-1)$ and the multiplicity $2\leq \mu\leq m$ of $X$ is given by $\mu=m+1-\frac{\ell_i\cdot p^{n-i}-1}{e}$; 
  \item[\rm(ii)] $\chi_X=\sum_{z=0}^{l}\chi_z +  \Xi(W,i)$  in case $l$ is even, $\chi_{x^{}_0}>0$, $e\mid(\ell_i\cdot p^{n-i}-1)$ and the multiplicity $2\leq \mu\leq m$ of $X$ is given by $\mu=\frac{\ell_i\cdot p^{n-i}-1}{e}+1$; 
  \item[\rm(iii)] $\chi_X=\sum_{z=0}^{l}\chi_z +  \Xi(W,i)$ in case $l$ is odd, $\chi_{x^{}_0}<0$, $e\mid\ell_i$ and the multiplicity $2\leq \mu\leq m$ of $X$ is given by $\mu=\frac{\ell_i\cdot p^{n-i}}{e} +1$;
  \item[\rm(iv)] $\chi_X=\sum_{z=0}^{l}\chi_z +  \overline{\Xi(W,i)}$ in case $l$ is even, $\chi_{x^{}_0}<0$, $e\mid\ell_i$ and the multiplicity $2\leq \mu\leq m$ of $X$ is given by $\mu=m+1- \frac{\ell_i\cdot p^{n-i}}{e} $.
\end{itemize}
\smallskip
\item[\rm\bf(3)]  If $X$ corresponds to the  path
  $$
\xymatrix@R=0.0000pt@C=30pt{	
		{_{\chi_{x^{}_0}}}&{_{\chi^{}_{\Lambda}}}\\
		{\Circle } \ar@<0.3ex>[r]^{E_1}  &{\CIRCLE}\ar@<0.3ex>[l]^{E_2}  
}
$$
where the direction is $\varepsilon=(-1,1)$ and  $\chi^{}_{\Lambda}$ is a leaf of $\sigma(\bB)$, \smallskip then:
\begin{itemize}
  \item[\rm(i)]  $\chi_X=\overline{\Xi(W,i)}$ in case $\chi^{}_{\Lambda}>0$, $e\mid(\ell_i\cdot p^{n-i}-1)$ and the multiplicity $2\leq \mu\leq m-1$ of $X$ is given by $\mu=m-\frac{\ell_i\cdot p^{n-i}-1}{e}$;
  \item[\rm(ii)] $\chi_X=\Xi(W,i)$ in case $\chi^{}_{\Lambda}<0$, $e\mid\ell_i$ and the multiplicity $2\leq \mu\leq m-1$ of $X$ is given by  $\mu=\frac{\ell_i\cdot p^{n-i}}{e}$. 
\end{itemize}
\smallskip
\item[\rm\bf(4)]  If  $X$ corresponds to the  path
$$  \xymatrix@R=0.0000pt@C=30pt{	
    &{_{\chi_{x^{}_0}}}&{_{\chi_{x^{}_1}}}&{_{\chi_{x^{}_l}}}&{_{\chi^{}_{\Lambda}}}\\
      {\Circle }  &{\Circle } \ar@<0.3ex>[r]^{E_1}  \ar@<0.3ex>[l]^{E_{s}}  &{\Circle }\ar@<0.3ex>[l]^{E_{s-1}}  \ar@{.}[r]&{\Circle }\ar@<0.3ex>[r]^{E_{l+1}}&{\CIRCLE}\ar@<0.3ex>[l]^{E_{l+2}}
}
$$
where $l\geq 0$, the successor of $E_1$ around $\chi_{x^{}_0}$ is $E_s$, the direction is $\varepsilon=(1,1)$, \smallskip then:
\begin{itemize}
  \item[\rm(i)] $\chi_X=\sum_{z=0}^{l}\chi_z +  \overline{\Xi(W,i)}$ in case $l$ is odd, $\chi_{x^{}_0}>0$, $e\mid(\ell_i\cdot p^{n-i}-1)$ and the multiplicity $2\leq \mu\leq m$ of $X$ is given by $\mu=m+1-\frac{\ell_i\cdot p^{n-i}-1}{e}$; 
  \item[\rm(ii)] $\chi_X=\sum_{z=0}^{l}\chi_z +  \Xi(W,i)$  in case $l$ is even, $\chi_{x^{}_0}>0$, $e\mid(\ell_i\cdot p^{n-i}-1)$ and the multiplicity $2\leq \mu\leq m$ of $X$ is given by $\mu=\frac{\ell_i\cdot p^{n-i}-1}{e}+1$; 
  \item[\rm(iii)] $\chi_X=\sum_{z=0}^{l}\chi_z +  \Xi(W,i)$ in case $l$ is odd, $\chi_{x^{}_0}<0$, $e\mid\ell_i$ and the multiplicity $2\leq \mu\leq m$ of $X$ is given by $\mu=\frac{\ell_i\cdot p^{n-i}}{e} +1$;
  \item[\rm(iv)] $\chi_X=\sum_{z=0}^{l}\chi_z +  \overline{\Xi(W,i)}$ in case $l$ is even, $\chi_{x^{}_0}<0$, $e\mid\ell_i$ and the multiplicity $2\leq \mu\leq m$ of $X$ is given by $\mu=m+1- \frac{\ell_i\cdot p^{n-i}}{e} $.
\end{itemize}
\smallskip
\item[\rm\bf(5)]  
If $X$ corresponds to the  path
$$  
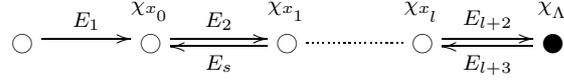
\xymatrix@R=0.0000pt@C=30pt{	
   {}    &{_{\chi_{x^{}_0}}}&{_{\chi_{x^{}_1}}}&{_{\chi_{x^{}_l}}}&{_{\chi^{}_{\Lambda}}}\\
      {\Circle } \ar@<0.3ex>[r]^{E_1}  &{\Circle }  \ar@<0.3ex>[r]^{E_{2}}  &{\Circle }\ar@<0.3ex>[l]^{E_{s}}  \ar@{.}[r]&{\Circle }\ar@<0.3ex>[r]^{E_{l+2}}&{\CIRCLE}\ar@<0.3ex>[l]^{E_{l+3}}
}
$$
where $l\geq 0$, the successor of $E_1$ around $\chi_{x^{}_0}$ is $E_s$, the direction is $\varepsilon=(-1,-1)$, \smallskip then: 
\begin{itemize}
  \item[\rm(i)] $\chi_X=\sum_{z=0}^{l}\chi_z +  \overline{\Xi(W,i)}$ in case $l$ is odd, $\chi_{x^{}_0}>0$, $e\mid(\ell_i\cdot p^{n-i}-1)$ and the multiplicity $2\leq \mu\leq m$ of $X$ is given by $\mu=m+1-\frac{\ell_i\cdot p^{n-i}-1}{e}$; 
  \item[\rm(ii)] $\chi_X=\sum_{z=0}^{l}\chi_z +  \Xi(W,i)$  in case $l$ is even, $\chi_{x^{}_0}>0$, $e\mid(\ell_i\cdot p^{n-i}-1)$ and the multiplicity $2\leq \mu\leq m$ of $X$ is given by $\mu=\frac{\ell_i\cdot p^{n-i}-1}{e}+1$; 
  \item[\rm(iii)] $\chi_X=\sum_{z=0}^{l}\chi_z +  \Xi(W,i)$ in case $l$ is odd, $\chi_{x^{}_0}<0$, $e\mid\ell_i$ and the multiplicity $2\leq \mu\leq m$ of $X$ is given by $\mu=\frac{\ell_i\cdot p^{n-i}}{e} +1$;
  \item[\rm(iv)] $\chi_X=\sum_{z=0}^{l}\chi_z +  \overline{\Xi(W,i)}$ in case $l$ is even, $\chi_{x^{}_0}<0$, $e\mid\ell_i$ and the multiplicity $2\leq \mu\leq m$ of $X$ is given by $\mu=m+1- \frac{\ell_i\cdot p^{n-i}}{e} $.
\end{itemize}
\smallskip
\item[\rm\bf(6)] If $X$ corresponds to the  path
$$ 
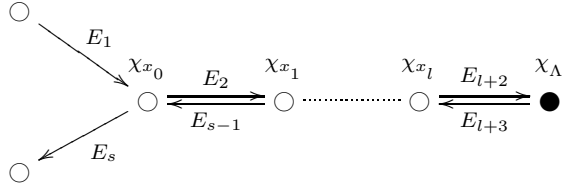
\xymatrix@R=0.0000pt@C=30pt{
 	&& &\\
	{\Circle }\ar[ddr]^{E_{1}} & & &  \\
		&{_{\chi_{x^{}_0}}}&{_{\chi_{x^{}_1}}}&{_{\chi_{x^{}_l}}}&{_{\chi^{}_{\Lambda}}} \\
		&{\Circle }\ar[dddl]^{\:\:E_{s}} \ar@<0.3ex>[r]^{E_2}&{\Circle }\ar@<0.3ex>[l]^{E_{s-1}}\ar@{.}[r]&{\Circle }\ar@<0.3ex>[r]^{E_{l+2}}&{\CIRCLE}\ar@<0.3ex>[l]^{E_{l+3}}\\
		&& &\\
		&& &\\
		{\Circle }& & & \\
		&& &
	}
$$
where $l\geq 0$,  the successor of $E_1$ around $\chi_{x^{}_0}$ is $E_s$,  the direction is $\varepsilon=(-1,1)$, \smallskip then:
\begin{itemize}
  \item[\rm(i)] $\chi_X=\sum_{z=0}^{l}\chi_z +  \overline{\Xi(W,i)}$ in case $l$ is odd, $\chi_{x^{}_0}>0$, $e\mid(\ell_i\cdot p^{n-i}-1)$ and the multiplicity $2\leq \mu\leq m$ of $X$ is given by $\mu=m+1-\frac{\ell_i\cdot p^{n-i}-1}{e}$; 
  \item[\rm(ii)] $\chi_X=\sum_{z=0}^{l}\chi_z +  \Xi(W,i)$  in case $l$ is even, $\chi_{x^{}_0}>0$, $e\mid(\ell_i\cdot p^{n-i}-1)$ and the multiplicity $2\leq \mu\leq m$ of $X$ is given by $\mu=\frac{\ell_i\cdot p^{n-i}-1}{e}+1$; 
  \item[\rm(iii)] $\chi_X=\sum_{z=0}^{l}\chi_z +  \Xi(W,i)$ in case $l$ is odd, $\chi_{x^{}_0}<0$, $e\mid\ell_i$ and the multiplicity $2\leq \mu\leq m$ of $X$ is given by $\mu=\frac{\ell_i\cdot p^{n-i}}{e} +1$;
  \item[\rm(iv)] $\chi_X=\sum_{z=0}^{l}\chi_z +  \overline{\Xi(W,i)}$ in case $l$ is even, $\chi_{x^{}_0}<0$, $e\mid\ell_i$ and the multiplicity $2\leq \mu\leq m$ of $X$ is given by $\mu=m+1- \frac{\ell_i\cdot p^{n-i}}{e} $.
\end{itemize}
\smallskip
\item[\rm\bf(7)]  If  $X$ corresponds to the  path
$$ 
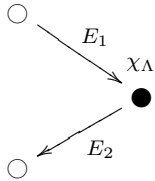
\xymatrix@R=0.0000pt@C=30pt{
               & \\	
		{\Circle }\ar[ddr]^{E_{1}} &  \\
		&{_{\chi_\Lambda}} \\
		&{\CIRCLE}\ar[dddl]^{\:\:E_{2}}\\
		&\\
		&\\
		{\Circle }&  \\
		& 
}
$$
where  the successor of $E_1$ around $\chi^{}_{\Lambda}$ is $E_2$ and  the direction is $\varepsilon=(-1,1)$, \smallskip then:
\begin{itemize}
  \item[\rm(i)]  $\chi_X=\overline{\Xi(W,i)}$ in case $\chi^{}_{\Lambda}>0$, $e\mid(\ell_i\cdot p^{n-i}-1)$ and the multiplicity $1\leq \mu\leq m-1$ of $X$ is given by $\mu=m-\frac{\ell_i\cdot p^{n-i}-1}{e}$;
  \item[\rm(ii)] $\chi_X=\Xi(W,i)$ in case $\chi^{}_{\Lambda}<0$, $e\mid\ell_i$ and the multiplicity $1\leq \mu\leq m-1$ of $X$ is given by  $\mu=\frac{\ell_i\cdot p^{n-i}}{e}$. 
\end{itemize}
\end{itemize}  
In all drawings of the paths the vertices $\chi_{x^{}_0},\ldots,{\chi_{x^{}_l}}\in\Irr'(\bB)$. 
\end{enumerate}
\end{thm}
\bigskip

\begin{rem}
To simplify, we say that the trivial source module $X$ has type ${\bf(2)}$ (resp. {\bf(3)},  {\bf(4)}, {\bf(5)},  {\bf(6)},  {\bf(7)}) if $X$ corresponds to a path of type {\bf(2)}, (resp. {\bf(3)},  {\bf(4)}, {\bf(5)},  {\bf(6)},  {\bf(7)})
in the statement of Theorem~\ref{thm:main}(b). We also note that this labelling agrees with the labelling of  \cite[Theorem~5.3]{HL19}.
\end{rem}

 \begin{proof}
 We shall go through the classification of the trivial source $\bB$-modules with vertex~$D_i$ provided by \cite[Theorem~5.3]{HL19}.
 Let $Y:=f(X)$ be the Green correspondent of $X$ in $\bb$.  Write $\Psi_Y=d_0\theta_{x_0}$ for some $1\leq x_0\leq e$ and some $d_0\in\{0,1\}$ for the non-exceptional part of $\chi_Y$ and 
$\Xi_Y= \theta_{\Lambda'}$ with $\Lambda'\subseteq\Lambda$ for the exceptional part of~$\chi_Y$.  
For each module occurring  in \cite[Theorem~5.3]{HL19}, we determine both the non-exceptional part $\Psi_X$ and  the exceptional part $\Xi_X$ of~$\chi_X$ from $\chi_Y$ as follows.
 \begin{enumerate}
\item[\rm(a)] If $e=1$, then $\bB$ is uniserial and there is a unique trivial source $\bB$-module $X$ with vertex~$D_i$. Also, more precisely, $\chi_Y=d_0\theta_{1}+\chi^{}_{\Lambda'}$ and
 $\chi_X$ must also have the form
$\chi_X=d_0'\chi_1+\Xi_{X}$ for some $d_0'\in\{0,1\}$. Hence,
$$\ell(Y)=d_0+|\Lambda'|\qquad\text{ and }\qquad\ell(X)=d_0'+\langle \Xi_X,\Xi_X\rangle_G\,.$$
By Proposition~\ref{prop:PerfIso}, either $\Xi_X=\chi^{}_{\Lambda'}$ or $\Xi_X=\chi^{}_{\Lambda\setminus\Lambda'}$, hence $\langle \Xi_X,\Xi_X\rangle_G\in\{|\Lambda'|,m-|\Lambda'|\}$. 
Now, by \cite[Theorem~5.3(a)]{HL19} there are two cases to distinguish for $X$. 
\begin{itemize}
  \item[$\cdot$]Case 1: $\chi_1>0$.  Then, it follows from \cite[Theorem~5.3(a) and its proof]{HL19} that 
  $$\ell(Y)=\ell(X)=\ell_i\cdot p^{n-i}\,.$$
By the above, the only possibility is $\Xi_X=\chi^{}_{\Lambda'}$ and $d_0'=d_0$, i.e. $\Psi_X=d_0\chi_1$. 
  \item[$\cdot$]Case 2: $\chi_1<0$.  Then by  \cite[Theorem~5.3(a) and its proof]{HL19}, 
  $$\ell(\Omega(Y))=\ell(X)=p^n-\ell_i\cdot p^{n-i}\,.$$
  Now, as the unique PIM of $\bb$ affords the character $\theta_1+\theta_{\Lambda}$,  the cotrivial source module $\Omega(Y)$ affords the character
  $$\chi_{\Omega(Y)}=(1-d_0)\theta_1+\theta_{\Lambda\setminus\Lambda'}$$ and it follows that the only possibility is  $\Xi_X=\chi^{}_{\Lambda\setminus\Lambda'}$ and $d_0'=1-d_0$, i.e.  $\Psi_X=(d_0-1)\chi_1$. 
\end{itemize}
\smallskip
Now, by Corollary~\ref{cor:tsb}(b), $\chi^{}_{\Lambda'}=\Xi(W,i)$, whereas $\chi^{}_{\Lambda\setminus\Lambda'}=\overline{\Xi(W,i)}$. By Corollary~\ref{cor:tsb}(a) yields $d_0=1$ if $t(i)$ is odd and $d_0=0$ if $t(i)$  is \medskip even.
\item[\rm(b)] We can now go through the classification of the trivial source $\bB$-modules with vertex $D_i$ provided by \cite[Theorem~5.3(b)]{HL19}.
To begin with, if $X$ has vertex $D$ and $W=k$, then $X$ is a hook and the claim follows from Lemma~\ref{lem:PIMhooksSimples}. \\
Thus, from now on we assume that $X$ has type {\bf(2)}, {\bf(3)},  {\bf(4)}, {\bf(5)},  {\bf(6)} or  {\bf(7)}. First of all, in all cases the non-exceptional part $\Psi_X$ of $\chi_X$ is given by \cite[Theorem~A.1(d)]{HL19}, namely $\Psi_X=\sum_{z=0}^{l}\chi_z$ if $X$ is of type  {\bf(2)}, {\bf(4)}, {\bf(5)} or {\bf(6)}, whereas $\Psi_X=0$ if $X$ is of type {\bf(3)} or {\bf(7)}. Therefore, it remains to compute the exceptional part $\Xi_X$ of $\chi_X$. Now, \cite[Theorem~A.1(d)]{HL19} also provides us with the number of constituents of $\Xi_X$, namely
$$\langle \Xi_X,\Xi_X\rangle_G
=
\begin{cases}
 \mu-1     & \text{if $X$ corresponds to a path of type {\bf(2)}, {\bf(4)}, {\bf(5)} or {\bf(6)}};\\
   \mu   & \text{if $X$ corresponds to a path of type {\bf(3)} or {\bf(7)}}.
\end{cases}$$
Let $Y:=f(X)$ be the Green correspondent of $X$ in $\bb$.  Write $\Psi_Y=d_0\theta_{x_0}$ for some $1\leq x_0\leq e$ and some $d_0\in\{0,1\}$ for the non-exceptional part of $\chi_Y$ and 
$\Xi_Y= \theta_{\Lambda'}$ with $\Lambda'\subseteq\Lambda$ for the exceptional part of~$\chi_Y$. By Lemma~\ref{lem:cardLambda'}, the number of constituents of $\Xi_Y$ is 
$$|\Lambda'|=\frac{\ell_i\cdot p^{n-i}-d_0}{e}\,.$$ 
Now, by Proposition~\ref{prop:PerfIso} there are two possibilities for $\Xi_X$. 
First, $\Xi_X=\chi^{}_{\Lambda'}$ if and only if $\langle\Xi_X,\Xi_X\rangle_G=|\Lambda'|$. Hence by the above
$$
\Xi_X=\chi^{}_{\Lambda'}\,\,\Leftrightarrow \,\, \mu=
\begin{cases}
  \frac{\ell_i\cdot p^{n-i}-d_0}{e} +1    & \text{if $X$ is of type {\bf(2)}, {\bf(4)}, {\bf(5)} or {\bf(6)}};\\
  \frac{\ell_i\cdot p^{n-i}-d_0}{e}    & \text{if $X$ is of type {\bf(3)} or {\bf(7)}}.
\end{cases}
$$
Second,  $\Xi_X=\chi^{}_{\Lambda\setminus\Lambda'}$ if and only if  $\langle \Xi_X,\Xi_X\rangle_G=|\Lambda\setminus\Lambda'|=m-|\Lambda'|$. 
Hence by the above
$$
\phantom{XXX}\Xi_X=\chi^{}_{\Lambda\setminus\Lambda'}\,\,\Leftrightarrow \,\, \mu=
\begin{cases}
  m+1-\frac{\ell_i\cdot p^{n-i}-d_0}{e}     & \text{if $X$ is of  type {\bf(2)}, {\bf(4)}, {\bf(5)} or {\bf(6)}};\\
  m-\frac{\ell_i\cdot p^{n-i}-d_0}{e}    & \text{if $X$ is of type {\bf(3)} or {\bf(7)}}.
\end{cases}
$$
In addition, by Corollary~\ref{cor:tsb}(b), $\chi^{}_{\Lambda'}=\Xi(W,i)$, whereas $\chi^{}_{\Lambda\setminus\Lambda'}=\overline{\Xi(W,i)}$.
Finally, we note that  by Corollary~\ref{cor:tsb}(a), we have $d_0=1$ if and only if $t(i)$ is even, which by construction happens if and only if $e\mid (\ell_i\cdot p^{n-i}-d_0)$ and $d_0=0$ if and only if $t(i)$ is odd, which by construction happens if and only if $e\mid \ell_i$. \\
This data together with the classification theorem 
{\color{black}\cite[Theorem~5.3(b)]{HL19} } 
yields the following form for \medskip $\Xi_X$. 
\begin{itemize}
\item[{\bf1.}] Types {\bf(2)}, {\bf(4)}, {\bf(5)} and {\bf(6)} all work identically. By 
{\color{black}\cite[Theorem~5.3(b)]{HL19} } 
there are four cases to \medskip distinguish. 
\begin{itemize}
\item[Case 1:]  $X$ is such that $l$ is odd, $\chi_{x^{}_0}>0$, $e\mid(\ell_i\cdot p^{n-i}-1)$ and the multiplicity $2\leq \mu\leq m$ of $X$ is given by $\mu=m+1-\frac{\ell_i\cdot p^{n-i}-1}{e}$. \\In this case it follows from the above that $d_0=1$ and $\Xi_X=\chi^{}_{\Lambda\setminus\Lambda'}=\overline{\Xi(W,i)}$.
\item[Case 2:] $X$ is such that  $l$ is even, $\chi_{x^{}_0}>0$, $e\mid(\ell_i\cdot p^{n-i}-1)$ and the multiplicity $2\leq \mu\leq m$ of $X$ is given by $\mu=\frac{\ell_i\cdot p^{n-i}-1}{e}+1$. \\In this case it follows from the above that $d_0=1$ and $\Xi_X=\chi^{}_{\Lambda'}=\Xi(W,i)$.
\item[Case 3:] $X$ is such that $l$ is odd, $\chi_{x^{}_0}<0$, $e\mid\ell_i$ and the multiplicity $2\leq \mu\leq m$ of $X$ is given by $\mu=\frac{\ell_i\cdot p^{n-i}}{e} +1$. \\In this case it follows from the above that $d_0=0$ and $\Xi_X=\chi^{}_{\Lambda'}=\Xi(W,i)$.
\item[Case 4:] $X$ is such that $l$ is odd, $\chi_{x^{}_0}<0$, $e\mid\ell_i$ and the multiplicity $2\leq \mu\leq m$ of $X$ is given by $\mu=m+1- \frac{\ell_i\cdot p^{n-i}}{e}$. \\In this case it follows from the above that $d_0=0$ and $\Xi_X=\chi^{}_{\Lambda\setminus\Lambda'}=\overline{\Xi(W,i)}$.
\end{itemize}
\bigskip
\item[\bf{2.}]Type {\bf(3)}:  By 
{\color{black}\cite[Theorem~5.3(b)]{HL19} } 
there are two cases to \medskip distinguish. 
\begin{itemize}
\item[Case 1:]  $X$ is such that  $\chi^{}_{\Lambda}>0$, $e\mid(\ell_i\cdot p^{n-i}-1)$ and the multiplicity $2\leq \mu\leq m-1$ of $X$ is given by $\mu=m-\frac{\ell_i\cdot p^{n-i}-1}{e}$. \\In this case it follows from the above that $d_0=1$ and $\Xi_X=\chi^{}_{\Lambda\setminus\Lambda'}=\overline{\Xi(W,i)}$.
\item[Case 2:]  $X$ is such that   $\chi^{}_{\Lambda}<0$, $e\mid\ell_i$ and the multiplicity $2\leq \mu\leq m-1$ of $X$ is given by  $\mu=\frac{\ell_i\cdot p^{n-i}}{e}$.  \\In this case it follows from the above that $d_0=0$ and $\Xi_X=\chi^{}_{\Lambda'}=\Xi(W,i)$.
\end{itemize}
\bigskip
\item[\bf{3.}]Type {\bf(7)}:  By 
{\color{black}\cite[Theorem~5.3(b)]{HL19} } 
there are two cases to \medskip distinguish. 
\begin{itemize}
\item[Case 1:]  $X$ is such that  $\chi^{}_{\Lambda}>0$, $e\mid(\ell_i\cdot p^{n-i}-1)$ and the multiplicity $1\leq \mu\leq m-1$ of $X$ is given by $\mu=m-\frac{\ell_i\cdot p^{n-i}-1}{e}$. \\In this case it follows from the above that $d_0=1$ and $\Xi_X=\chi^{}_{\Lambda\setminus\Lambda'}=\overline{\Xi(W,i)}$.
\item[Case 2:]  $X$ is such that   $\chi^{}_{\Lambda}<0$, $e\mid\ell_i$ and the multiplicity $1\leq \mu\leq m-1$ of $X$ is given by  $\mu=\frac{\ell_i\cdot p^{n-i}}{e}$.  \\In this case it follows from the above that $d_0=0$ and $\Xi_X=\chi^{}_{\Lambda'}=\Xi(W,i)$.
\end{itemize}
\end{itemize}
\end{enumerate}
 \end{proof}

\begin{rem}
In  \cite{Tak12} M. Takahashi computed the ordinary characters afforded by Scott modules in groups with cyclic Sylow $p$-subgroups, where the inertial index of the principal block is greater than one. 
Scott modules all belong to the principal block and correspond to paths of the form
$$
\xymatrix@R=0.0000pt@C=30pt{	
		{_{\chi_{x_{0}}}}&{_{\chi_{x_1}}}&{_{\chi_{x_l}}}&{_{\chi^{}_{\Lambda}}}\\
		{\Circle } \ar@<0.3ex>[r]^{k}  &{\Circle }\ar@<0.3ex>[l]^{k}  \ar@{.}[r]    &{\Circle } \ar@<0.3ex>[r]^{E_{l+1}} & {\CIRCLE} \ar@<0.3ex>[l]^{E_{l+2}}
}
$$
with $\chi_{x^{}_0}=1_G>0$ and $E_1=E_s=k$. For the principal block, $W=k$, because it is isomorphic to a source of the trivial $kG$-module. Hence $\ell_i=1$ and $e\mid (p^{n-i}-1)$ for each $1\leq i\leq n$. Thus the Scott module with vertex $D_i$ correspond to a module of type {\bf (2)} in Theorem~\ref{thm:main}(b) with $\chi_{x^{}_0}>0$ and $e\mid (p^{n-i}-1)$. 
\end{rem}

\bigskip
\bigskip
\bigskip
\bigskip


\noindent\textbf{Acknowledgments.} 
The authors are grateful to Markus Linckelmann for thorough explanations on Chapters 9 to 11 of his book \cite{LinckBook}.
The first author would like to acknowledge funding by DFG SFB-TRR 195 and thank the Department of Mathematics of the TU Kaiserslautern
for its hospitality when he visited there several times in 2018 and 2019. The second author also would like to thank the department of mathematics of the university of Chiba for its hospitality and M. Takahashi for explanations on her results on Scott modules during her visit in 2014.

\bigskip\bigskip

\bigskip\bigskip




\providecommand{\bysame}{\leavevmode\hbox to3em{\hrulefill}\thinspace}
\providecommand{\MR}{\relax\ifhmode\unskip\space\fi MR }
\providecommand{\MRhref}[2]{%
  \href{http://www.ams.org/mathscinet-getitem?mr=#1}{#2}
}
\providecommand{\href}[2]{#2}

\bigskip
\bigskip


\end{document}